\RequirePackage{rotating}
\documentclass[a4paper]{amsart}
\usepackage{latexsym}\usepackage{ifthen}\usepackage[leqno]{amsmath}
\usepackage{enumerate}\usepackage{calc}\usepackage{hyphenat}
\usepackage{anyfontsize}
\usepackage{moresize}
\usepackage{mathtools}
\usepackage{bbm}
\usepackage{amstext,amsbsy,amsopn,amsthm,amsgen, amsmath}
\usepackage{amsfonts,amscd,amsxtra,upref}
\usepackage{graphicx,varwidth}
\usepackage{wrapfig}
\usepackage{pdfsync}
\usepackage[margin=1in,headheight=13.6pt]{geometry}

\usepackage{multicol}
\usepackage{float}
\usepackage{array,amsmath}

\usepackage{afterpage}
\usepackage{comment}
\usepackage{expl3,xparse}

\usepackage{multirow}
\usepackage[table]{xcolor} 
\usepackage{array,booktabs}

\usepackage{slashbox}
\usepackage{tabularx}
\usepackage{stackrel}
\usepackage{caption, float}
\usepackage{hhline}
\usepackage{dirtytalk}

\usepackage{pdfpages}

\usepackage{picture,slashbox}

\usepackage{epstopdf}
\usepackage{dsfont}
\usepackage{hyperref}
\usepackage{mathrsfs}\usepackage{euscript}\usepackage{amssymb}
\swapnumbers
\theoremstyle{plain}

\usepackage[
backend=bibtex,
style=alphabetic,
natbib=true,
sortlocale=en_US,
url=true, 
doi=true,
eprint=true,
]{biblatex}
\newtheorem{Thm}{Theorem}[section]

\newtheorem{Lem}[Thm]{Lemma}

\newtheorem{Cor}[Thm]{Corollary}
\newtheorem{Pro}[Thm]{Proposition}

\theoremstyle{definition}
\newtheorem{Def}[Thm]{Definition}

\newtheorem*{thm28}{{\rm 2.8.}~Theorem}
\theoremstyle{remark}
\newtheorem{Rem}[Thm]{Remark}
\numberwithin{equation}{section}

\newcommand{\ITE}[3]{\ifthenelse{#1}{#2}{#3}}
\ITE{\isundefined{\texorpdfstring}}{\newcommand{\texorpdfstring}[2]{#1}}{}

\usepackage{scalerel}
\usepackage{stackengine,wasysym}
\newcommand\tylda[1]{\ThisStyle{%
		\setbox0=\hbox{$\SavedStyle#1$}%
		\stackengine{-.1\LMpt}{$\SavedStyle#1$}{%
			\stretchto{\scaleto{\SavedStyle\mkern.2mu\AC}{.5150\wd0}}{.6\ht0}%
		}{O}{c}{F}{T}{S}%
}}

\newcommand\tyldaa[1]{\ThisStyle{%
		\setbox0=\hbox{$\SavedStyle#1$}%
		\stackengine{-.1\LMpt}{$\SavedStyle#1$}{%
			\stretchto{\scaleto{\SavedStyle\mkern -1.5mu\AC}{.5150\wd0}}{.5\ht0}%
		}{O}{c}{F}{T}{S}%
}}

\newcommand*\quot[2]{{^{\textstyle #1}\big/_{\textstyle #2}}}

\newcommand{\myData}[1][]{
\author[D.\ Burek]{Dominik Burek}
\address{D.\ Burek{}\\{}%%
Faculty of Mathematics and Computer Science\\{}Jagiellonian University,
\L{}ojasiewicza 6\\{}30-348 Krak\'{o}w\\{}Poland}
\email{dominik.burek@uj.edu.pl}
}

\newcommand{\myCurrentData}[1][]{
\address{Current address: 
Leibniz Universit{\"a}t Hannover\\{}Fakult{\"a}t f{\"u}r Mathematik und Physik \\{} Welfengarten 1\\{}30167 Hannover}}

\newcommand{\ZZ}{\mathbb{Z}}

\newcommand{\PP}{\mathbb{P}}
\newcommand{\OO}{\mathcal{O}}
\newcommand{\FF}{\mathbb{F}}
\newcommand{\NN}{\mathbb{N}}

\newcommand{\diag}{\operatorname{diag}}

\newcommand{\QQ}{\mathbb{Q}}
\newcommand{\et}{\operatorname{\acute{e}t}}
\newcommand{\QQl}{\mathbb{Q}_{l}}

\newcommand{\Fr}{\operatorname{Frob}}

\newcommand{\id}{\operatorname{id}}
\newcommand{\Fix}{\operatorname{Fix}}
\newcommand{\Frob}{\operatorname{Frob}}

\newcommand{\orb}{\operatorname{orb}}

\newcommand{\age}{\operatorname{age}}

\newcommand{\overbar}[1]{\mkern 1.5mu\overline{\mkern-1.5mu#1\mkern-1.5mu}\mkern 1.5mu}
\newcommand\restrict[1]{\raisebox{-.5ex}{$|$}_{#1}}
\renewcommand{\bar}{\overbar}

\def\restrict#1{\raise-.5ex\hbox{\ensuremath|}_{#1}}

\ExplSyntaxOn
\NewDocumentCommand{\lam}{m}{%
	\group_begin:
	\tl_set:Nn \l_tmpa_tl {#1}
	\hbox_set:Nn \l_tmpa_box {\ensuremath{\l_tmpa_tl}}
	\dim_set:Nn \l_tmpa_dim {\box_ht:N \l_tmpa_box}
	\dim_add:Nn \l_tmpa_dim {\box_dp:N \l_tmpa_box}
	\dim_add:Nn \l_tmpa_dim {0.3em}
	\dim_set:Nn \l_tmpb_dim {1.2\l_tmpa_dim}
	\tl_replace_all:Nnn \l_tmpa_tl {,} {,\allowbreak}
	\ifmmode%
	\text{\fontsize{\l_tmpa_dim}{\l_tmpb_dim}\selectfont}
	\else%
	{\fontsize{\l_tmpa_dim}{\l_tmpb_dim}\selectfont}
	\fi
	\ensuremath{\l_tmpa_tl}
	\ifmmode%
	\text{\fontsize{\l_tmpa_dim}{\l_tmpb_dim}\selectfont}
	\else%
	{\fontsize{\l_tmpa_dim}{\l_tmpb_dim}\selectfont}
	\fi
	\group_end:
}
\ExplSyntaxOff

\begin{document}
	\title[Higher dimensional analogon of Borcea-Voisin CY's, Hodge numbers and $L$-functions]{Higher dimensional analogon of Borcea-Voisin Calabi-Yau manifolds, their Hodge numbers and $L$-functions}
	\myData
	%\myCurrentData
	
	\begin{abstract} 
		 We construct a series of examples of Calabi-Yau manifolds in an arbitrary dimension and compute the main invariants. In particular, we give higher dimensional generalization of Borcea-Voisin Calabi-Yau threefolds.
		 
		 We give a method to compute a local zeta function using the Frobenius morphism for orbifold cohomology introduced by Rose. We compute Hodge numbers of the constructed examples using orbifold Chen-Ruan cohomology.\end{abstract}
	
	%\thanks{This work was partially supported by the grant 346300 for IMPAN from the Simons Foundation and the matching 2015-2019 Polish MNiSW fund.}
	\maketitle
	
	%\tableofcontents
	
	\section{Introduction}
	
%	Borcea observed that the mirror resolution it should hold even in higher dimension. There is a delicate issue concerning the (so-called crepant) resolution of the quotient. Chiodo, Veniani and Kalashnikov did study this in terms of orbifolds (their work aims at generalising the mirror symmetry claim). Their work does not concern the existence of crepant resolution. 
%Various methods have been used to 

This paper is devoted to the construction of a higher dimensional generalization of Borcea-Voisin Calabi-Yau threefolds. This work is motivated by the limited range of examples of Calabi-Yau manifolds in dimensions at least 4 which are neither complete intersection in projective spaces nor hypersurfaces in toric space.  

The first series of examples of Calabi-Yau threefolds with a mirror symmetry phenomenon were constructed independently by C.\ Borcea (\cite{BorceaC}) and C.\ Voisin (\cite{V}). This construction can be regarded both as a method of constructing new Calabi-Yau $3$-folds and also as a way to construct mirror pairs. The Borcea-Voisin construction involves a $K3$ surface $S$ and an elliptic curve $E$ equipped with a not period preserving involutions $\gamma_{S}\colon S\to S$, $\alpha_{E}\colon E\to E$, respectively. The quotient $(S\times E)/(\gamma_{S}\times \alpha_{E})$ has a~resolution of singularities which is a Calabi-Yau threefold. A mirror Calabi-Yau threefold is obtained by the same construction using a mirror $K3$ surface coming from Nikulin's classification of $K3$ surfaces with non-symplectic involution (\cite{Nikulin66}). Moreover Borcea shows the coherence of this construction with Batyrev's construction for hypersurfaces in toric varieties (\cite{Batyrevdual}). He observed that the mirror resolution should hold also in higher dimensions and worked out several cases in four and five dimension. Mirror symmetry for $K3$ surfaces with a non-symplectic involution was generalized to $K3$ surfaces with purely non-symplectic automorphisms (\cite{CamparinPriddis}, \cite{AMBS}). % but he did not verify the
%assertion about the Yukawa coupling.

Various ideas have been used to give generalizations of the Borcea-Voisin construction, however not many of them giving Calabi-Yau manifold have results in mirror symmetry. C. Vafa and E. Witten studied similar constructions in \cite{VafaWitten} as a resolution of singularities of $(E\times E\times E)/(\ZZ_2\oplus \ZZ_2)$ and $(E\times E\times E)/(\ZZ_3\oplus \ZZ_3)$, where $E$ is an elliptic curve and $\ZZ_2$ (resp. $\ZZ_3$) acts as involution (resp. automorphism of order 3). This approach leads to abstract physical models studied by L. Dixon, J. Harvey, C. Vafa, E. Witten in \cite{DHVW1, DHVW2}. More generally, quotients of products of tori by an abelian extension of a~group $\ZZ_2\oplus \ZZ_2$ were classified by R. Donagi, and K. Wendland in \cite{DonagiWendland}. Calabi-Yau threefolds arising via this construction include Vafa-Witten threefolds, Borcea-Voisin threefolds and Schoen's threefolds (\cite{Schoen}).%DilliesToroidal, DonagiFaraggi,

A.~Cattaneo and A. Garbagnati  
in \cite{C} observed that the Borcea-Voisin construction can be generalized in a different way. Denote by $S_{d}$ a $K3$ surface admitting a purely non-symplectic automorphism $\gamma_{d}$ of order $d$ and let $E_d$ be an elliptic curve with not period preserving automorphism $\alpha_{d}$ of order $d.$ Then taking on the first factor an automorphism $\phi_{d}$ and on the second factor we take $\alpha_{d}^{d-1}$ we consider the quotient $(S_d\times E_d)/(\gamma_{d}\times \alpha_{d}^{d-1})$. Their paper is organized based upon the order of the not period preserving automorphisms $\gamma_{d}$ and $\alpha_{d}$ which must be either $2,$ $3$, $4,$ or $6.$ In each case, the authors explicitly constructed resolutions of singularities of $(S_d\times E_d)/(\gamma_{d}\times \alpha_{d}^{d-1})$ and study the properties of the resulting Calabi-Yau 3-folds. In the case of order three the resulting threefolds are already known (see J. C. Rohde \cite{Rohde}, \cite{Rohde2} and J. Dillies \cite{DilliesGen}). Incidentally Dillies considered also quotient of $S\times S'$ by the cyclic group $\ZZ_{p}$, where $S,$ $S'$ are $K3$ surfaces admitting purely non-symplectic automorphisms of prime order $p>2,$ and prove that in most cases a crepant resolution does not exists. Very recently Cattaneo, Garbagnati and M. Penegini in \cite{CGP} considered two $K3$ surfaces $S_1$ -- double cover of a del Pezzo surface and $S_2$ admitting an~ elliptic fibration $\pi\colon S_2\to \PP^2.$ Then the quotient of $(S_1\times S_2)/(i_1\times i_2)$, where $i_{1}$ is the involution of $S_1$ coming from the~double cover structure and $i_{2}$ is the elliptic involution admits a crepant resolution of singularities.

It is natural to ask whether the generalization of the original Borcea-Voisin construction to higher dimension by allowing more factors works in higher dimensions. A. Chiodo, E. Kalashnikov and D. C. Veniani in \cite{Kalach} studied this in terms of semi Calabi-Yau orbifolds by taking
$\displaystyle \prod_{i=1}^{2n}\left(\mathcal{X}_{i}^{\lor}\right)/(\sigma_{i}, \ldots, \sigma_{2n}),$ where $\mathcal{X}_{i}$ is a Calabi-Yau stack with involution $\sigma_{i}.$ Nevertheless their work does not concern the existence of crepant resolutions and mostly aims at generalizing the mirror symmetry aspects.

%Incidentaly only few series of examples of higher dimensional Calabi-Yau manifolds are known. Except Cynk-Hulek construction there is also double cover of \PP^{5} branched along 12 planes ( ) and   
On the other hand S. Cynk and K. Hulek in \cite{CH} studied where one can extend the classical construction of Borcea-Voisin as far as taking the product of elliptic curves and resolving the quotient in a crepant way. The authors proved using an iterative argument that for an elliptic curve $E_{d}$ admitting a~not period preserving automorphism of order $d=2,3,4,$ the quotient of $E_{d}^n$, by a group isomorphic to $\ZZ_{d}^{n-1}$ and preserving the volume form, admits a crepant resolution of singularities. The missing degree $d=6$ cannot be handled by the original approach, it was settled in \cite{Burii3} using a toric crepant resolution of singularities. 

The present paper investigates higher dimensional construction of Calabi-Yau manifolds which generalizes the classical Borcea-Vosin construction by merging ideas of Cynk-Hulek and Cattaneo-Garbagnati. Precisely, we prove the following:
%2.6
\begin{thm28}\label{constr} Let $E_d$ be an elliptic curve with a not period preserving automorphism of order $d$ and $S_{d}$ be a $K3$ surface with purely non-symplectic automorphism of order $d\in\{2,3,4,6\}.$ Consider the following group $$G_{d,n}:=\{(m_{1},m_{2}, \ldots, m_{n})\in \ZZ_{d}^{n} \colon m_{1}+m_{2}+\ldots + m_{n}=0\}\simeq \ZZ_{d}^{n-1}$$ acting naturally on $S_d\times E_{d}^{n-1}$. Then the quotient variety $(S_d\times E_{d}^{n-1})/G_{d,n}$ admits a crepant resolution of singularities $Y_{d,n}$. Consequently $Y_{d,n}$ is an $(n+1)$-dimensional Calabi-Yau variety.\end{thm28}
%consider the quotient of $S_{d}\times E_{d}^{n-1}$ by the group preserving canonical bundle and isomorphic to $\ZZ_{d}^{n-1}.$
Our advance uses both ideas of taking many factors (in spirit of Cynk-Hulek) and $(d-1)$-th power of an automorphism on a second factor (inspired by Cattaneo-Garbagnati).

In order to study a crepant resolution of singularities of manifolds $(S_d\times E_{d}^{n-1})/G_{d,n}$ we shall study Calabi-Yau manifolds constructed in greater generality, as a resolution of singularities of $$(X_{1}\times X_{2}\times \ldots \times X_{n})/\ZZ_{d}^{n-1},$$ where $X_{i}$ denotes a variety of Calabi-Yau type (i.e. $K_{X_{i}}= \OO_{X_{i}}$) with $pnpp$ (purely not period preserving automorphism, see Def. \ref{non-per}) $\phi_{i,d}\colon X_i\to X_i$ of order $d$, where $d=2,$ $3,$ $4$ or $6.$ We prove that if  $\Fix(\phi_{i,d}^{k})$ is a smooth divisor for $k\mid d$ and $i=2,\ldots, n$ and $\phi_{1,d}$ satisfies a certain assumption $\hyperref[assump] {\textup{\textbf{A\textit{d}}}}$ then there exists a crepant resolution $\mathcal{X}_{d,n}$ which is a Calabi-Yau type manifold (Thm. \ref{constr}). Moreover we prove that if $X_{1},$ $X_{2},$ $\ldots,$ $X_{n}$ are Calabi-Yau manifolds, then $\mathcal{X}_{d,n}$ is also Calabi-Yau manifold (Cor. 3.7). 
%In the special case when $X_{1}$ are elliptic curves we obtain the construction given in \cite{CH}, \cite{Burii3}.

%Our construction largely increases the amount of higher dimensional Calabi-yau threefolds with feasible Hodge numbers and zeta functions (twierdzenia wypisac) 

We shall give a formula for the Hodge numbers of $\mathcal{X}_{d,n}$ using Chen-Ruan orbifold cohomology theory (\cite{CR}). In \cite{Burii3} we have computed Hodge numbers in the special case of $X_{i}=E_{d}$ (for $i=1,2,\ldots, n$) by a careful study of the action of $\phi_{i,d}$ on $\Fix(\phi_{i,d}^k)$ for $k\mid d$ and $i=1,\ldots, n.$ This approach does not generalize to the current setup, in the present paper we use characteristic polynomials of the eigenspaces of the action of $\phi_{i,d}$
on the cohomology groups of $X_{i}$ (Thm. \ref{wzorek}). The novelty of our approach consists in the improvement of the use of orbifold formula (no orbit counting of the group action) and possibility implementation of the Hodge numbers formula in \texttt{MAPLE}. Moreover, due to the wide generality of the~considered manifolds, the formula can be used for many other examples. 

We shall study in more detail a crepant resolution of the quotient $(S_{d}\times E_{d}^{n-1})/\ZZ_{d}^{n-1}.$
Comparing the Euler characteristics of the constructed manifolds computed from our formula for Hodge numbers with the stringy Euler number we obtain new relations among invariants of $K3$ surfaces with purely non-symplectic automorphism of order $3,$ $4$ and $6$ (section 4). 

In section 5 we propose an effective method to compute local zeta functions of constructed generalized Calabi-Yau manifolds of Borcea-Voisin type. By the Weil conjecture (analogon of the Riemann hypothesis) the zeta function of a smooth projective variety defined over $\FF_q$ has the following form 
$$Z_{q}(t)=\frac{P_{1,q}(t)\cdot P_{3,q}(t)\cdot \ldots \cdot P_{2d-1,q}(t)}{P_{0,q}(t)\cdot P_{2,q}(t)\cdot \ldots\cdot P_{2d,q}(t)},$$ where $P_{i,q}(t):=\det\left(1-\Fr_{q}^{*}t\mid H^{i}_{\et}(X,\QQl)\right)$ is the characteristic polynomial of the Frobenius Galois action.
Using Frobenius polynomials $P_{i,p}$ of the~reduction $\bar{X}_{p}$ of a variety $X$ over $\QQ$ modulo a good prime $p$, we define the $L$-series as:
$$L_{}\left(H^{i}_{\et}\left(\overline{X},\QQl\right), s\right):=(*)\prod_{p\in\mathcal{P}}\frac{1}{P_{i,p}(p^{-s})},$$
where $\mathcal{P}$ is the set of primes of good reduction and $(*)$ denotes suitable Euler factors for the primes of bad reduction. In the special case of rigid Calabi-Yau threefolds, the $L$-series is the $L$-series of a modular form (\cite{YuiGou}, \cite{Dieulefait}). Y. Goto, R. Livn{\'e} and N. Yui in \cite{GOTOYUI} computed $L$-function of Borcea-Voisin Calabi-Yau threefolds over $\QQ.$ 

In dimensions greater than 3 Cynk and Hulek in \cite{CH} proved modularity and determined the~corresponding cusp forms of the constructed Calabi-Yau manifolds. We are not aware of any further examples of Calabi-Yau manifold of dimension greater than 3 with explicitly computed $L$-function.

Using the description of the Frobenius action on the orbifold cohomology (\cite{Rose}) we prove that the~zeta function of a generalized Borcea-Voisin manifold is a tensor product of factors associated to natural submotives for fixed loci of elements of the group $\ZZ_{d}^{n-1}$ (Thm. \ref{zetawzorek}). Our approach gives an effective way of computation of $L$-functions of considered manifolds and can be implemented for any dimension.
 
As the zeta function is of arithmetic nature we  cannot get a general formula as in the case of Hodge numbers. Our approach allows us to compute the zeta function in explicit examples, in section 5.1 we give details of computations in two particular cases. The first example involves a family of $K3$ surfaces with known zeta function (\cite{OnoPeniston}), the second one is a Calabi-Yau manifold with particularly interesting shape of the Hodge diamond. 
	
%The aim of this paper is to give an effective methods of computation of the zeta function of a quotient variety of type $\quot{X_{1}\times X_{2}\times \ldots \times X_{n}}{\ZZ_{d}^{n-1}},$ where $X_{i}$ are Calabi-Yau type manifolds with purely non-symplectic $\ZZ_{d}$-action. 
	
%	Another possibility is to take a quotient of a product of two $K3$ surfaces by a finite group. Such fourfolds were studied by J. Dillies in \cite{Dil2}. F. Reidegald divided $S\times \PP^{1},$ where $S$ is a $K3$ surface, by a cyclic group of order 3. He also found a desingularization of such quotients (see \cite{Reidegald}).
	
\subsection*{Acknowledgments} This paper is a part of author's PhD thesis. I am deeply grateful to my advisor S\l{}awomir Cynk for his enormous help. I would like to thank Matthias Sch{\"u}tt for helpful suggestions and comments. The author is supported by the National Science Center of Poland grant no. 2019/33/N/ST1/01502 and National Science Center of Poland grant no. 2020/36/T/ST1/00265. 
% We would like to thank anonymous referee whose comments improved the presentation of our paper.

The final part of this work was conducted during the stay at the Leibniz Universit{\"a}t in Hannover. I would like to thank the institute for hospitality. Finally I would like to thank the anonymous referee for remarks and advices that substantially improve the presentation of this paper. 
	
\section{Kummer type Calabi-Yau manifolds}

\subsection{Borcea-Voisin construction}

The first example of a family of Calabi-Yau threefolds which is symmetric with respect to Mirror Symmetry was constructed independently by C.\ Borcea (\cite{BorceaC}) and C.~Voisin (\cite{V}). Their construction involves an elliptic curve $E$ and a $K3$ surface $S$ with a not period preserving involution $\alpha_{E}\colon E\to E$ and a non-symplectic involution $\gamma_{S}\colon S\to S$. They also computed the Hodge numbers of the~constructed Calabi-Yau threefold.
	
\begin{Thm}[\cite{BorceaC,V}]{} The quotient $(S\times  E)/(\gamma_{S}\times \alpha_{E})$ has a crepant resolution of singularities $X$ which is a Calabi-Yau threefold. Moreover
		$$h^{1,1}(X)=11+5N-N'\quad \textrm{and}\quad h^{2,1}(X)=11+5N'-N,$$
		where $N$ is the number of curves in $\Fix(\gamma_{S})$ and $N'$ is the sum of their genera.
\end{Thm}

The classification of $K3$ surfaces with non-symplectic involution was given by Nikulin (\cite{NikulinClass}).
From Nikulin's classification it follows that a mirror of Borcea-Voisin Calabi-Yau threefold is a Borcea-Voisin Calabi-Yau threefold associated to the mirror $K3$ surfaces $\tylda{S}.$ 
%for any $K3$ surface $S$ with non-symplectic involution $\alpha_{S}$ which fixes $N$ curves with sum of genera equal to $N',$ there exists a complementary surface $S'$ and its non-symplectic involution $\alpha_{S}'$ with $N'$ fixed curves with sum of genera equal to $N.$ Thus the pair $$\tylda{(E\times S)/(\alpha_{E}\times \alpha_{S})}\quad\textup{and}\quad \tylda{(E\times S')/(\alpha_{E}\times \alpha_{S'})}$$ is a mirror pair.
In \cite{C} A. Cattaneo and A. Garbagnati generalized the Borcea-Voisin construction allowing a non-symplectic automorphisms of a $K3$ surfaces of higher degrees i.e. $3,$ $4$ and $6.$

\begin{Thm}[\cite{C}]
	Let $S_{d}$ be a $K3$ surface admitting a purely non-symplectic automorphism $\gamma_{d}$ of order $d=3,4,6.$ Let $E_{d}$ be an elliptic curve admitting a not period preserving automorphism $\alpha_{d}$ of order $d.$ Then the~quotient $(S_{d}\times E_{d})/(\gamma_{d}\times \alpha_{d}^{n-1})$ is a singular variety which admits a crepant resolution of singularities $\tyldaa{(S_{d}\times E_{d})/(\gamma_{d}\times \alpha_{d}^{d-1})}$. In particular $\tyldaa{(S_{d}\times E_{d})/(\gamma_{d}\times \alpha_{d}^{d-1})}$ is a Calabi-Yau threefold.
\end{Thm}

The authors gave a detailed crepant resolution and computed the Hodge numbers of the resulting algebraic varieties. For all possible orders they computed the Hodge numbers of these varieties and constructed elliptic fibrations on them, in \cite{Burii1} we gave much simpler derivations of formulas for Hodge numbers using Chen-Ruan cohomology \ref{Formula}. %Their computations are quite technical and rely on a detailed study of crepant resolutions of threefolds. 

\subsection{Generalisation}

%Uwaga: Although this notion depends on a particular choice of primitive root of unity, existence of a pnpp automorhpis of a Calabi-Yau manifold is independed of this choice. 
%If X is a K3 surface \omega_{X} is a symplectic form, so pnpp is direct generalization of a purely non-symplectic automorhps to an arbitrary dimension. 
%For $d=2,$ $3,$ $4$ and $6,$ 
Let us fix a primitive $d$-th ($d\in\{2,3,4,6\})$ root of unity $\zeta_{d}$ once and for all. Let $X$ be a complex smooth projective manifold of dimension $n$ with trivial canonical bundle (we shall call them of \textit{Calabi-Yau type}). Any automorphism $\eta\colon X\to X$ of order $d$ satisfies 
$$\eta_{}^{*}\left(\omega_{X_{}}\right)=\zeta_d^k\omega_{X_{}},$$ where $\omega_{X_{}}\in H^{n,0}(X)$ denotes a non-zero canonical form (\textit{period} of $X$) and $k$ is a positive integer. 

\begin{Def} \label{non-per} An automorphism $\eta\colon X\to X$ of order $d$ is called \textit{purely not period preserving} ($pnpp$) iff $\eta_{}^{*}\left(\omega_{X_{}}\right)=\zeta_d\omega_{X_{}}.$ \end{Def} 

\begin{Rem} 1) The above definition depends on the choice of $\zeta_{d}$ but it is independent of the choice of $\omega_{X}.$

2) If $X$ is a $K3$ surface then $\omega_{X}$ is a symplectic form, so $pnpp$ automorphism is a direct generalization of a purely non-symplectic automorphism to an arbitrary dimension.  \end{Rem}
%For $d=2,$ $3,$ $4$ and $6,$\end{Rem}
% Replacing $\eta$ with it's appropriate power we can assume that $\eta_{}^{*}\left(\omega_{X_{}}\right)=\zeta_d\omega_{X_{}}.$  
Consider the following assumptions:
\label{assump}
Condition $\label{a3}\bf{A}3$ (for $d=3$)
\begin{enumerate}
%	\item the fixed point locus $\Fix(\eta_{2})$ of $\eta_{2}$ is a disjoint union of smooth divisors, in particular $\eta_{2}$ has linearisation of the form $(\zeta_{3}, 1, 1, \ldots, 1)$ near any point of $\Fix(\eta_{2})$,
	\item $\Fix(\eta_{})$ is a disjoint union of submanifolds of codimension at most 2. In particular $\eta_{}$ has linearisation of the form
	\begin{itemize} 
		\item $(\zeta_{3}^{2}, 1, 1, \ldots, 1)$ near a component of codimension one of $\Fix(\eta_{})$,
		\item $(\zeta_{3}^{}, \zeta_{3}, 1,1, \ldots, 1)$ near a component of codimension two of $\Fix(\eta_{})$.
	\end{itemize}
\end{enumerate}
%\begin{Pro}[\cite{CH}, Proposition 3.1]\label{pro3} Under the above assumptions, the quotient variety $(X_1\times X_2)/\ZZ_3$ has a resolution of singularities $X$ which is a Calabi-Yau manifold. The manifold $X$ admits a $\ZZ_3$-action which satisfies the same assumptions as for $X_2.$ \end{Pro}

%Let $X_{1}, X_{2}$ be two Calabi-Yau manifolds with automorphisms $\eta_{i}\colon X_{i}\to X_{i}$ (for $i=1,2$) of order $
%	4$ such that
%	$$\eta_{1}^{*}\left(\omega_{X_{1}}\right)=\zeta_{4}\omega_{X_{1}}\quad \textup{and} \quad \eta_{2}^{*}\left(\omega_{X_{2}}\right)=\zeta_{4}^{3}\omega_{X_{2}}, $$ where $\omega_{X_{i}}$ denotes a chosen generator of $H^{n,0}(X_i),$ for $i=1,2.$ Suppose that
Condition $\bf{A}4$ (for $d=4$)
	\begin{enumerate}
%		\item the fixed point locus $\Fix(\eta_{2})$ of $\eta_{2}$ is a disjoint union of smooth divisors, in particular $\eta_{2}$ has linearisation of the form $(\zeta_{4}, 1, 1, \ldots, 1)$ near any point of $\Fix(\eta_{2})$,
		\item $\Fix(\eta_{})$ is a disjoint union of submanifolds of codimension at most 3. In particular $\eta_{}$ has linearisation of the form
		\begin{itemize} 
			\item $(\zeta_{4}^{3}, 1, 1, \ldots, 1)$ near a component of codimension one of $\Fix(\eta_{})$,
			\item $(\zeta_{4}^{}, \zeta_{4}^{2}, 1,1, \ldots, 1)$ near a component of codimension two of $\Fix(\eta_{})$,
			\item $(\zeta_{4}^{}, \zeta_{4}^{},\zeta_{4}, 1,1, \ldots, 1)$ near a component of codimension three of $\Fix(\eta_{})$.
		\end{itemize}
	\end{enumerate}
	
%begin{Pro}[\cite{CH}, Proposition 4.1]\label{pro4} Under the above assumptions, the quotient variety $(X_1\times X_2)/\ZZ_4$ has a resolution of singularities $X$ which is a Calabi-Yau manifold. The manifold $X$ admits a $\ZZ_4$-action which satisfies the same assumptions as for $X_2.$ \end{Pro}

%In \cite{CH} a crepant resolution of $X_{n,d}$ for $d=2,3,4$ was constructed by an iterated approach. 
%In \cite{Burii3} we gave the missing construction for elliptic curves admitting automorphisms of order 6.

%Let $X_{1}, X_{2}$ be two Calabi-Yau manifolds with automorphisms $\eta_{i}\colon X_{i}\to X_{i}$ (for $i=1,2$) of order $6$ such that
%$$\eta_{1}^{*}\left(\omega_{X_{1}}\right)=\zeta_{6}\omega_{X_{1}}\quad \textup{and} \quad \eta_{2}^{*}\left(\omega_{X_{2}}\right)=\zeta_{6}^{5}\omega_{X_{2}}, $$ where $\omega_{X_{i}}$ denotes a chosen generator of $H^{n,0}(X_i),$ for $i=1,2.$

Condition $\bf{A}6$ (for $d=6$)
\begin{enumerate} 
%	\item the fixed point locus $\Fix(\eta_{2})$ of $\eta_{2}$ is a disjoint union of smooth divisors, in particular $\eta_{2}$ has linearisation of the form $(\zeta_{6}, 1, 1, \ldots, 1)$ near any point of $\Fix(\eta_{2})$,
	\item $\Fix(\eta_{})$ is a disjoint union of submanifolds of codimension at most 3. In particular $\eta_{}$ has linearisation of the form
	\begin{itemize} 
		\item $(\zeta_{6}^{5}, 1, 1, \ldots, 1)$ near a component of codimension one of $\Fix(\eta_{})$,
		\item $(\zeta_{6}^{4}, \zeta_{6}, 1,1, \ldots, 1)$ or $(\zeta_{6}^{3}, \zeta_{6}^{2}, 1,1, \ldots, 1)$ near a component of codimension two of $\Fix(\eta_{})$,
	\end{itemize}
%	\item $\Fix(\eta_{2}^{2})\setminus \Fix(\eta_{2})$ is a disjoint union of smooth divisors in particular $\eta_{2}^{2}$ has linearisation $(\zeta_{3}^{}, 1, 1, \ldots, 1)$ along any component of $\Fix(\eta_{2}^{2})\setminus \Fix(\eta_{2}),$
%	\item $\Fix(\eta_{2}^{3})\setminus \Fix(\eta_{2})$ is a disjoint union of smooth divisors in particular $\eta_{2}^{3}$ has linearisation $(-1, 1, 1, \ldots, 1)$ along any component of $\Fix(\eta_{2}^{3})\setminus \Fix(\eta_{2}),$
	\item $\Fix(\eta_{}^{2})\setminus \Fix(\eta_{})$ is a disjoint union of smooth submanifolds of codimension at most $2,$ so $\eta_{}^2$ has linearisation of the form $(\zeta_{6}^4, 1, 1, \ldots, 1)$ or $(\zeta_{6}^2,\zeta_{6}^2, 1, 1, \ldots, 1)$ along any component of $\Fix(\eta_{}^{2})\setminus \Fix(\eta_{})$,
	\item $\Fix(\eta_{}^{3})\setminus \Fix(\eta_{})$ is a disjoint union of smooth divisors, so $\eta_{}^3$ has linearisation of the form $(-1, 1, 1, \ldots, 1)$ along any component of $\Fix(\eta_{}^{3})\setminus \Fix(\eta_{})$,
	\item\label{zalozenie}the automorphism $\eta_{}$ has a local linearisation of the form $(\zeta_{6}^{2}, \zeta_{6}^{2},\zeta_{6}^{}, 1,1, \ldots, 1)$ along any codimensional $3$ component of $\Fix(\eta_{}).$
\end{enumerate}

We have the following:

\begin{Pro}[\cite{Burii3}, \cite{CH}]\label{pro6} Let $X_{1}$ and $X_{2}$ be projective manifolds of Calabi-Yau type and a $pnpp$ automorphism $\eta_{i}\colon X_i\to X_i$ of order $d,$ where $d\in\{2,3,4,6\}.$ Assume moreover, that the automorphism $\eta_1$ satisfies condition $\hyperref[assump]{\textup{\textbf{A\textit{d}}}}$ and $\Fix(\eta_{2}^k)$ for $k\mid d$ is a smooth divisor in $X_2.$ Then the quotient $(X_1\times X_2)/(\eta_1\times\eta_2^{d-1})$ admits a crepant resolution of singularities $X$. Moreover the automorphism $\id\times \eta_2$ lifts to a $pnpp$ automorphism of $X$ satisfying $\hyperref[assump]{\textup{\textbf{A\textit{d}}}}$.\end{Pro}

\newcommand*{\LargerCdot}{\raisebox{-0.25ex}{\scalebox{1.5}{$\cdot$}}}
Let $X_{1},$ $X_{2},$ $\ldots,$ $X_{n}$ be projective manifolds of Calabi-Yau type and let $\phi_{i,d}\colon X_i\to X_i$ be a $pnpp$ automorphism of order $d,$ where $d\in\{2,3,4,6\}.$
Consider the following group 
$$G_{d,n}:=\{(m_{1},m_{2}, \ldots, m_{n})\in \ZZ_{d}^{n}\colon m_{1}+m_{2}+\ldots + m_{n}=0\}\simeq \ZZ_{d}^{n-1}$$ which acts on $X_{1}\times X_{2}\times \ldots \times X_{n}$ as $$(m_1, m_2, \ldots, m_{n})_{\LargerCdot}(x_{1}, x_{2}, \ldots, x_{n}) =\left(\phi_{1,d}^{m_1}(x_{1}), \phi_{2,d}^{m_2}(x_{2}), \ldots, \phi_{n,d}^{m_n}(x_{n})\right).$$

Using Prop. \ref{pro6} we can prove by an easy induction the following theorem:

\begin{Thm}\label{glowne}
Let $X_1,$ $X_2,$ $\ldots,$ $X_n$ be projective manifolds of Calabi-Yau type with $pnpp$ automorphisms $\phi_{i,d}\colon X_i\to X_i$ of order $d,$ where $d\in\{2,3,4,6\}.$
Assume moreover, that automorphism $\phi_{1,d}$ satisfies condition $\hyperref[assump]{\textup{\textbf{A\textit{d}}}}$ and $\Fix(\phi_{i,d}^k)$ for $k\mid d$ is a smooth for $i=2,3,\ldots, n.$ Then the quotient $$(X_{1}\times X_{2}\times \ldots \times X_{n})/G_{d,n}$$ admits a crepant resolution $\mathcal{X}_{d,n}.$ Moreover the automorphism $ \phi_{1,d}\times \id$ lifts to a $pnpp$ automorphism of $\mathcal{X}_{d,n}$. \end{Thm}

The manifold $\mathcal{X}_{d,n}$ has a trivial canonical bundle and hence it is a Calabi-Yau manifold if the outer Hodge numbers vanish $$H^{i,0}(\mathcal{X}_{d,n})=0,\quad \textup{$0<i<\dim \mathcal{X}_{d,n}$}.$$ As we see in Corollary \ref{zerowe} this condition is immediately satisfied when $X_{1},$ $X_{2},$ $\ldots,$ $X_{n}$ are Calabi-Yau manifolds.

%\hyperref[preli]{Chapter 2}.

In \cite{Burii3}, \cite{CH} the special case of the above theorem, when $X_{i}=E_{d}$ was considered.  

\begin{Thm}[\cite{CH}, \cite{Burii3}]\label{ch} If $d=2,3,4,6$ then there exists a crepant resolution $$ \tyldaa{E_{d}^{n}/G_{d,n}}\rightarrow E_{d}^{n}/G_{d,n}.$$ Consequently, $X_{d,n}:=\tyldaa{E_{d}^{n}/G_{d,n}}$ is an $n$-dimensional Calabi-Yau manifold. \end{Thm}

%In \cite{Burii3} we used another approach for computing Hodge numbers of varieties $X_{d,n}$ based on a systematic study of orbits of the action of $G_{d,n}.$ The method is very complex and we were not able to generalize it to the case of Calabi-Yau manifold other than elliptic curves. 

%In further study we consider much more general construction. Let $X_{i}$ be a variety of Calabi-Yau type with purely non-symplectic automorphism $\phi_{i,d}\colon X_i\to X_i$ of order $d$ i.e. for any $\omega_{X_{i}}\in H^{n_i, 0}(X_{i}, \mathcal{O}_{X_{i}})$ and the fixed $d$-th root of unity $\zeta_{d}$: $$\phi_{i,d}(\omega_{X_{i}})=\zeta_{d} \omega_{X_{i}},$$ for $i=1,2,\ldots,n.$

%\begin{Thm}\label{glowne} Let $d\in \{2,3,4,6\}$ and assume that $\Fix(\phi_{i,d})$ is a divisor for $i\in\{2,3,\ldots, n\}$ while $X_{1}$ satisfies condition $\hyperref[assump]{\textup{\textbf{A\textit{d}}}}.$ Then $$(X_{1}\times X_{2}\times \ldots \times X_{n})/G_{d,n}$$ has a crepant resolution of singularities  $\mathcal{X}_{d,n}.$ Moreover the automorphism $ \eta_1\times \id$ lifts to a purely non-symplectic automorphism of $\mathcal{X}_{d,n}$. \end{Thm}

In \cite{Burii3} we have computed Hodge numbers of the resulting varieties by using the orbifold formula \ref{Formula} and a careful study of the action of $\phi_{i,d}$ on $\Fix(\phi_{i,d}^k)$. In the present paper we shall propose new approach which is much more general and allows for new applications.

One can see that if $X_{1}$ equals $S_d$ -- a $K3$ surface with purely non-symplectic automorphism of order $d,$ then the condition $\hyperref[assump]{\textup{\textbf{A\textit{d}}}}$ is automatically satisfied. Therefore as a special case of Theorem \ref{glowne} we have the following:

\begin{Thm}\label{constr} Let $E_d$ be an elliptic curve with a not period preserving automorphism of order $d$ and $S_{d}$ be a $K3$ surface with purely non-symplectic automorphism of order $d\in\{2,3,4,6\}.$ Then the quotient variety $(S_d\times E_{d}^{n-1})/G_{d,n}$ admits a crepant resolution of singularities $Y_{d,n}$. Consequently $Y_{d,n}$ is an $(n+1)$-dimensional Calabi-Yau variety.\end{Thm}

\section{Hodge numbers of a finite quotient}

In \cite{Burii3} we determined Hodge numbers of varieties $X_{d,n}$ based on Chen-Ruan cohomology and systematic study of orbits of the action of $G_{d,n}.$ The method is very complex and we were not able to generalize it to the case of Calabi-Yau manifolds $Y_{d,n}.$ 

\subsection{Chen-Ruan cohomology}

Let us firstly recall a definition of an age of a matrix, which is important in the further context of this paper.

\begin{Def} For $G \in \operatorname{GL}_{n}(\mathbb{C})$ of order $m,$ let $e^{2\pi i a_{1}}, e^{2\pi i a_{2}}, \ldots, e^{2\pi i a_{n}}$ be eigenvalues of $G$ where $a_{1}, a_{2}, \ldots, a_{n}\in [0,1)\cap \mathbb{Q}$. The sum $a_1 + a_2+ \ldots + a_n$ is called the \textit{age} of $G$ and is denoted by $\age(G).$
\end{Def}
\begin{Rem} The age of $G$ is an integer if and only if $\det G=1$ i.e. $G \in \operatorname{SL}_{n}(\mathbb{C})$. \end{Rem}

In \cite{CR} W. Chen and Y. Ruan introduced new cohomology theory for orbifold. Let $X$ be a projective variety and $G$ be a finite group which acts on $X.$  

\begin{Def}\label{Formula} For a variety $X/G$ define the \textit{Chen-Ruan cohomology} by \begin{equation}\label{orbfor} H_{\textrm{orb}}^{i,j}(X/G):=\bigoplus_{[g]\in\textrm{Conj}(G)}\left(\bigoplus_{U\in \Lambda(g)} H^{i- \age(g),\; j-\age(g)}(U)\right)^{\textup{C}(g)}, \end{equation} where $\textrm{Conj}(G)$ is the set of conjugacy classes of $G$ (we choose a representative $g$ of each conjugacy class), $\textup{C}(g)$ is the centralizer of $g$, $\Lambda(g)$ denotes the set of irreducible connected components of the set fixed by $g\in G$ and $\age(g)$ is the age of the matrix of linearised action of $g$ near a point of $U.$ \par The dimension of $ H_{\textrm{orb}}^{i,j}(X/G)$ will be denoted by $h_{\text{orb}}^{i,j}(X/G).$\end{Def}

\begin{Rem} The definition makes sense since age is locally constant along each component of $\Fix(g).$ The components of $\Fix(g)$ are not necessarily invariant under the action of $\textup{C}(g)$, so in \ref{orbfor} we need to consider the action of $\textup{C}(g)$ on the inner direct sum (not on each summand separately).   \end{Rem}

An important feature of Chen-Ruan cohomology is the possibility of computing Hodge numbers of a crepant resolution of singularities of a quotient variety, without referring to an explicit construction of such a resolution. 

\begin{Thm}[\cite{TY}]{}\label{TY3}
	Let $G$ be a finite group acting on an algebraic smooth variety $X$. If there exists a crepant resolution $\tylda{X/G}$ of the variety $X/G,$ then the following equality holds $$h^{i,j}(\tylda{X/G})=h^{i,j}_{\orb}\left(X/G\right).$$
\end{Thm} In particular Hodge numbers of any crepant resolution do not depand on the specific resolution. For a systematic exposition of the orbifold Chen-Ruan cohomology see \cite{Leida}.

\subsection{Hodge numbers of $\mathcal{X}_{d,n}$}
\label{metoda}

Let $X_{i}$ be a variety of Calabi-Yau type of dimension $n_{i}$ with $pnpp$ automorphisms $\phi_{i,d}\colon X_i\to X_i$ of order $d$. Suppose that there exists a crepant resolution $\mathcal{X}_{d,n}$ of the quotient variety $$(X_{1}\times X_{2}\times \ldots \times X_{n})/\ZZ_{d}^{n-1},$$
then $\mathcal{X}_{d,n}$ is of Calabi-Yau type. The aim of this section is to give a formula for the Hodge number of $\mathcal{X}_{d,n}$ using Chen-Ruan cohomology. We shall use the following obvious lemma:

\begin{Lem}\label{lemaciko} Let $V_{1},$ $V_{2},$ $\ldots,$ $V_{n}$ be vector spaces over field of $k$ containing $\zeta_{d}$. Assume that $\alpha_{i}\in \textup{End}(V_{i})$ for $i=1,2,\ldots n$ is an automorphism of order $d.$ Then the fixed locus of $G_{d,n}$ acting on $V_{1}\times V_{2}\times \ldots \times V_{n}$ by 
	$$\alpha_{g}\colon G_{d,n}\ni g\mapsto \alpha_{1}^{g_1}\times \ldots \times \alpha_{n}^{g_n}\in \textup{End}(V_{1}\times V_{2}\times \ldots \times V_{n})$$ equals
	$$\bigoplus_{m=0}^{d-1}\left(V_{1}\right)_{\zeta_{d}^m}\otimes \ldots \otimes \left(V_{n}\right)_{\zeta_{d}^m},$$ where $\left(V_{i}\right)_{\zeta_{d}^m}$ denotes $m$-th eigenspace of $\ZZ_{d}$ action on $V_{i},$ for $i\in\{1,2,\ldots, n\}$ and $m\in \{0,1,\ldots, d-1\}.$ \end{Lem}

\begin{proof} Since the action of $G_{d,n}$ is diagonalizable we can consider only the tensor product of eigenvectors $v_{1}\otimes \ldots \otimes v_{n}$ where $\alpha_{i}(v_{i})=\zeta_{d}^{m_i}v_{i}.$ If $m_{i}\neq m_{j}$ for some $1\leq i<j\leq n$ then consider an element $$g_{ij}=\big(0,\ldots, \underbrace{1}_{i-\textup{th place}}, \ldots \underbrace{d-1}_{j-\textup{th place}},\ldots 1\big).$$ Since we have $$\alpha_{g_{ij}}(v_{1}\otimes \ldots \otimes v_{n})=\zeta_{d}^{m_i+(d-1)m_{j}}v_{1}\otimes \ldots \otimes v_{n}$$ the vector $v_{1}\otimes \ldots \otimes v_{n}$ is fixed by $\alpha_{g_{ij}}$ iff $d\mid m_i+(d-1)m_{j}$ which is equivalent to $m_{i}=m_{j}.$\end{proof}

Observe first that for $i\in\{1,2,\ldots, n\}$, $0\leq m_{i}\leq d-1$, the automorphism $\phi_{i,d}^{m_{i}}$ has local diagonalization near a point of $\Fix\left(\phi_{i,d}^{m_{i}}\right)$ of the following form $$\left(\zeta_{d}^{\alpha_{1}m_i}, \ldots, \zeta_{d}^{\alpha_{n_i}m_i}\right),$$ where $\zeta_{d}^{\alpha_{1}}\cdot\ldots \cdot \zeta_{d}^{\alpha_{n_{i}}} =\zeta_{d}$. Consequently $$\age\left(\phi_{i,d}^{m_{i}}\right)=\frac{m_{i}}{d}+\lambda_i,$$ where $\lambda_i$ is a non-negative integer, $\lambda_i<n_i$, moreover $\lambda_i$ is constant along every component of $\Fix(\alpha_{i}).$ 

\begin{Def} Define $$X_{i,m_i,\lambda_i}=\left\{x\in \Fix\left(\phi_{i,d}^{m_i}\right)\colon \age\left(\phi_{i,d}^{m_i}\right)=\frac{m_i}{d}+\lambda_i\;\; \textup{near } x\right\}$$ and
	
	$$F_{X_{i},m_i,j}(U,V)=\sum_{\lambda_{i}\geq 0}\sum_{0\leq p,q\leq \dim X_{i}}\dim_{\mathbb{C}}\left(H^{p,q}\left(X_{i, m_i,\lambda_{i}}\right)_{\zeta_{d}^j}\right)U^{p+\lambda_{i}}V^{q+\lambda_{i}},$$
where $H^{p,q}\left(X_{i, m_i,\lambda_{i}}\right)_{\zeta_{d}^j}$ denotes $\zeta_{d}^j$-eigenspace of the action of $\phi_{i,d}^{m_{i}}$ on the space $H^{p,q}\left(X_{i, m_i,\lambda_{i}}\right).$\end{Def}

We shall compute polynomials $F_{X_{i},m_i,j}$ only in the case of $X_{i}=E_{d}$ or $X_{i}=S_{d},$ then we shall denote them simply by $F_{E_{d},m_i,j}$ or $F_{S_{d},m_i,j}$.   

\begin{Thm}\label{wzorek} Assume that $X_{i}$ is a variety of Calabi-Yau type of dimension $n_{i}$ with $pnpp$ automorphisms  $\phi_{i,d}\colon X_i\to X_i$ of order $d$, $i=1,2,\ldots, n.$ Suppose that there exists a crepant resolution $\mathcal{X}_{d,n}$ of the quotient variety $$(X_{1}\times X_{2}\times \ldots \times X_{n})/G_{d,n}.$$
  Then \begin{equation} h^{p,q}\left(\mathcal{X}_{d,n}\right)= \sum_{j=0}^{d-1}\prod_{i=1}^{n}\Bigg(\sum_{m=0}^{d-1}\sqrt[d]{(UV)^{m}}\cdot F_{X_{i},m,j}(U,V) \Bigg)[U^{p}V^{q}], \end{equation} where $\mathcal{P}[U^{p}V^{q}]$ denotes the coefficient of a (Puiseaux) polynomial $\mathcal{P}$ in front of $U^{p}V^{q}.$ 
\end{Thm}

\begin{proof} By the Yasuda's Thm. \ref{TY3} $$H^{p,q}\left(\mathcal{X}_{d,n}\right)=H^{p,q}_{\textup{orb}}\left((X_{1}\times X_{2}\times \ldots \times X_{n})/\ZZ_{d}^{n-1}\right).$$ Now by the K{\"u}nneth formula for Hodge groups it follows that
	\bgroup\allowdisplaybreaks
	\begin{align*} H^{p,q}\left(\mathcal{X}_{d,n}\right)=\bigoplus_{m\in G_{n,d}}\Bigg(\bigoplus_{\mathfrak{\lambda}\in \NN^{n}}\bigoplus_{\substack{\mathfrak{p},\mathfrak{q}\in \NN^{n} \\ |\mathfrak{p}|=p-\frac{1}{d}|m|-|\mathfrak{\lambda}| \\ |\mathfrak{q}|=q-\frac{1}{d}|m|-|\mathfrak{\lambda}|}}H^{p_{1},q_{1}}\left(X_{1,m_{1}, \lambda_{1}}\right)\otimes \ldots \otimes H^{p_{n},q_{n}}\left(X_{n,m_{n}, \lambda_{n}}\right)\Bigg)^{G_{n,d}},
	\end{align*} 
	\egroup
	where $\mathfrak{\lambda}=(\lambda_{1}, \lambda_{2}, \ldots, \lambda_{n}),$ $\mathfrak{p}=(p_{1}, p_{2}, \ldots, p_{n}),$ $\mathfrak{q}=(q_{1}, q_{2}, \ldots, q_{n})$ and  $|\mathfrak{\lambda}|=\lambda_{1}+\lambda_{2}+ \ldots+ \lambda_{n}$, $|\mathfrak{p}|=p_{1}+p_{2}+ \ldots+ p_{n}$, $|\mathfrak{q}|=q_{1}+q_{2}+ \ldots+ q_{n}$. 
	
	Since 
	$\phi_{d,n}\left(X_{i,m_i,\lambda_i}\right)=X_{i,m_i,\lambda_i}$, the group $G_{n,d}$ acts separately on each summand of the inner sum, hence by the lemma \ref{lemaciko} we get 
	\bgroup\allowdisplaybreaks
	\begin{align*}
		H^{p,q}\left(\mathcal{X}_{d,n}\right)=\bigoplus_{m\in G_{n,d}}\bigoplus_{\mathfrak{\lambda}\in \NN^{n}}\bigoplus_{\substack{\mathfrak{p},\mathfrak{q}\in \NN^{n} \\ |\mathfrak{p}|=p-\frac{1}{d}|m|-|\mathfrak{\lambda}| \\ |\mathfrak{q}|=q-\frac{1}{d}|m|-|\mathfrak{\lambda}|}}\bigoplus_{j=0}^{d} H^{p_{1},q_{1}}(X_{1,m_{1}, \lambda_{1}})_{\zeta_{d}^j}\otimes \ldots \otimes H^{p_{n},q_{n}}(X_{n,m_{n}, \lambda_{n}})_{\zeta_{d}^j}.
	\end{align*}
	\egroup
	Taking dimensions and forming the generating function we get 
	\bgroup\allowdisplaybreaks
	%\resizebox{\linewidth}{!}{
	%	\begin{minipage}{\linewidth}
	\begin{align*}
		&\sum_{0\leq p,q\leq \dim \mathcal{X}_{d,n}}h_{\orb}^{p,q}\left(\mathcal{X}_{d,n}\right)U^pV^q=\sum_{0\leq p,q\leq \dim \mathcal{X}_{d,n}}\sum_{m\in G_{n,d}}\sum_{\mathfrak{\lambda}\in \NN^{n}}\sum_{j=0}^{d}\sum_{\substack{\mathfrak{p},\mathfrak{q}\in \NN^{n} \\|\mathfrak{p}|=p\\ |\mathfrak{q}|=q}}h^{p_{1}-\lambda_{1},q_{1}-\lambda_{1}}\left(X_{1,m_{1}, \lambda_{1}}\right)_{\zeta_d^j}\cdot \ldots \cdot \\&\cdot h^{p_{n}-\lambda_{n},q_{n}-\lambda_{n}}\left(X_{n,m_{n}, \lambda_{n}}\right)_{\zeta_d^j} U^{p}V^{q}\cdot (UV)^{\frac{m_1+\ldots + m_n}{d}}=\\&=\sum_{m\in G_{n,d}}\sum_{\mathfrak{\lambda}\in \NN^{n}}\sum_{j=0}^{d}\left(\sum_{p_{1},q_{1}\geq 0}h^{p_{1}-\lambda_{1},q_{1}-\lambda_{1}}\left(X_{1,m_{1}, \lambda_{1}}\right)_{\zeta_{d}^{j}}U^{p_{1}}V^{q_{1}}\cdot(UV)^{\frac{m_1}{d}}\right)\times \ldots \times \\&\times \left(\sum_{p_{n},q_{n}\geq 0}h^{p_{n}-\lambda_{n},q_{n}-\lambda_{n}}\left(X_{n,m_{n}, \lambda_{n}}\right)_{\zeta_{d}^{j}}U^{p_{n}}V^{q_{n}}\cdot(UV)^{\frac{m_n}{d}}\right)=\\&=
		\sum_{m\in G_{n,d}}\sum_{\mathfrak{\lambda}\in \NN^{n}}\sum_{j=0}^{d-1}\left(\sum_{p_{1},q_{1}\geq 0}h^{p_{1},q_{1}}\left(X_{1,m_{1}, \lambda_{1}}\right)_{\zeta_{d}^{j}}U^{p_{1}}V^{q_{1}}\right)(UV)^{\lambda_1+\frac{m_1}{d}}\cdot\ldots\cdot \\ &\cdot \left(\sum_{p_{n},q_{n}\geq 0}h^{p_{n},q_{n}}\left(X_{n,m_{n}, \lambda_{n}}\right)_{\zeta_{d}^{j}}U^{p_{n}}V^{q_{n}}\right)(UV)^{\lambda_n+\frac{m_n}{d}}.
	\end{align*}
	\egroup
	Enlarging the exterior sum $\displaystyle \sum_{m\in G_{n,d}}(\ldots)$ to $\displaystyle \sum_{m\in \ZZ_{d}^{n}}(\ldots)$ we introduce only terms with appropriate fractional powers of $U$ and $V$, hence $h^{p,q}\left(\mathcal{X}_{d,n}\right)$ is the coefficient in $U^{p}V^{q}$ of the following Puiseux polynomial:
	\begin{align*}&\sum_{m\in \ZZ_{d}^{n}}\sum_{\pmb{\lambda}\in \NN^{n}}\sum_{j=0}^{d-1}\left(\sum_{p_{1},q_{1}\geq 0}h^{p_{1},q_{1}}\left(X_{1,m_{1}, \lambda_{1}}\right)_{\zeta_{d}^{j}}U^{p_{1}}V^{q_{1}}\right)(UV)^{\lambda_1+\frac{m_1}{d}}\cdot\ldots\cdot \\ &\cdot \left(\sum_{p_{n},q_{n}\geq 0}h^{p_{n},q_{n}}\left(X_{n,m_{n}, \lambda_{n}}\right)_{\zeta_{d}^{j}}U^{p_{n}}V^{q_{n}}\right)(UV)^{\lambda_n+\frac{m_n}{d}}=
		\sum_{j=0}^{d-1}\prod_{i=1}^{n}\sum_{m=0}^{d-1}\sqrt[d]{(UV)^{m}}\cdot F_{X_{i},m_i,j}(U,V). \end{align*}

\end{proof}

From the above Thm. \ref{wzorek} it is particularly easy to compute the outer Hodge numbers $$H^{i,0}(\mathcal{X}_{d,n})=0,\quad \textup{$0<i<\dim \mathcal{X}_{d,n}$}.$$  
By Thm. 2.4 this implies:

\begin{Cor}\label{zerowe} If $X_{1},$ $X_{2},$ $\ldots,$ $X_{n}$ are Calabi-Yau manifolds then $\mathcal{X}_{d,n}$ is also a Calabi-Yau manifold. \end{Cor}

\subsection{Hodge numbers of $X_{d,n}$}

In this section we apply theorem \ref{wzorek} to the case $X_{i}=E_{d},$ reproving formulas from \cite{Burii3}. The proof given here is much shorter and cleaner. 

The tables below give polynomials $F_{E_d,k,j}(U,V)$ for $d=2,3,4,6.$ 
\begin{figure}[H]
	\centering 
	\begin{minipage}[b][5.5cm][b]{.40\textwidth}
		\centering
		\begin{table}[H]
\begin{center}
	\begin{varwidth}{\textheight}	\rowcolors{2}{gray!25}{white}
		\resizebox{0.3\textheight}{!}{
			\begin{tabular}{c||c||c||c||c||c||c}
				\hhline{~|~|~|~|~|~|~}
				\backslashbox{$k$}{$j$}  &
				\cellcolor{gray!30}0 & \cellcolor{gray!30}1 & \cellcolor{gray!30}2 &\cellcolor{gray!30} 3 & \cellcolor{gray!30}4 & \cellcolor{gray!30}5\\
				\hhline{=::=::=::=::=::=::=:}
				
				\setlength{\fboxrule}{5pt}\fcolorbox{white}{white}{0} & \setlength{\fboxrule}{5pt}\fcolorbox{white}{white}{$1+UV$} & \setlength{\fboxrule}{5pt}\fcolorbox{white}{white}{$U$}  & \setlength{\fboxrule}{5pt}\fcolorbox{white}{white}{0} &  \setlength{\fboxrule}{5pt}\fcolorbox{white}{white}{0} & \setlength{\fboxrule}{5pt}\fcolorbox{white}{white}{0} & \setlength{\fboxrule}{5pt}\fcolorbox{white}{white}{$V$}\\
				
				\setlength{\fboxrule}{5pt}\fcolorbox{gray!25}{gray!25}{1} & \setlength{\fboxrule}{5pt}\fcolorbox{gray!25}{gray!25}{1} & \setlength{\fboxrule}{5pt}\fcolorbox{gray!25}{gray!25}{0} & \setlength{\fboxrule}{5pt}\fcolorbox{gray!25}{gray!25}{0} & \setlength{\fboxrule}{5pt}\fcolorbox{gray!25}{gray!25}{0} & \setlength{\fboxrule}{5pt}\fcolorbox{gray!25}{gray!25}{0} & \setlength{\fboxrule}{5pt}\fcolorbox{gray!25}{gray!25}{0}\\
				
				\setlength{\fboxrule}{5pt}\fcolorbox{white}{white}{2} & \setlength{\fboxrule}{5pt}\fcolorbox{white}{white}{2} & \setlength{\fboxrule}{5pt}\fcolorbox{white}{white}{0}  & \setlength{\fboxrule}{5pt}\fcolorbox{white}{white}{0} &  \setlength{\fboxrule}{5pt}\fcolorbox{white}{white}{1} & \setlength{\fboxrule}{5pt}\fcolorbox{white}{white}{0} & \setlength{\fboxrule}{5pt}\fcolorbox{white}{white}{0}\\
				
				\setlength{\fboxrule}{5pt}\fcolorbox{gray!25}{gray!25}{3} & \setlength{\fboxrule}{5pt}\fcolorbox{gray!25}{gray!25}{2} & \setlength{\fboxrule}{5pt}\fcolorbox{gray!25}{gray!25}{0} & \setlength{\fboxrule}{5pt}\fcolorbox{gray!25}{gray!25}{1} & \setlength{\fboxrule}{5pt}\fcolorbox{gray!25}{gray!25}{0} & \setlength{\fboxrule}{5pt}\fcolorbox{gray!25}{gray!25}{1} & \setlength{\fboxrule}{5pt}\fcolorbox{gray!25}{gray!25}{0}\\
				
				\setlength{\fboxrule}{5pt}\fcolorbox{white}{white}{4} & \setlength{\fboxrule}{5pt}\fcolorbox{white}{white}{2} & \setlength{\fboxrule}{5pt}\fcolorbox{white}{white}{0}  & \setlength{\fboxrule}{5pt}\fcolorbox{white}{white}{0} &  \setlength{\fboxrule}{5pt}\fcolorbox{white}{white}{1} & \setlength{\fboxrule}{5pt}\fcolorbox{white}{white}{0} & \setlength{\fboxrule}{5pt}\fcolorbox{white}{white}{0}\\
				
				\setlength{\fboxrule}{5pt}\fcolorbox{gray!25}{gray!25}{5} & \setlength{\fboxrule}{5pt}\fcolorbox{gray!25}{gray!25}{1} & \setlength{\fboxrule}{5pt}\fcolorbox{gray!25}{gray!25}{0} & \setlength{\fboxrule}{5pt}\fcolorbox{gray!25}{gray!25}{0} & \setlength{\fboxrule}{5pt}\fcolorbox{gray!25}{gray!25}{0} & \setlength{\fboxrule}{5pt}\fcolorbox{gray!25}{gray!25}{0} & \setlength{\fboxrule}{5pt}\fcolorbox{gray!25}{gray!25}{0}\\
			\end{tabular}
		}
	\end{varwidth}
\vspace{3mm}
	\captionof{table}{$F_{E_6,k,j}(U,V)$}\label{Order6E}
\end{center}
\end{table}
	\end{minipage}
	\quad \quad \quad \quad \quad
	\begin{minipage}[b][5cm][b]{.40\textwidth}
		\centering
		\begin{table}[H]
	\begin{center}
		\begin{varwidth}{\textheight}	\rowcolors{2}{gray!25}{white}
			\resizebox{0.3\textheight}{!}{
				\begin{tabular}{c||c||c||c||c}
					\hhline{~|~|~|~|~}
					\backslashbox{$k$}{$j$}  &
					\cellcolor{gray!30}0 & \cellcolor{gray!30}1 & \cellcolor{gray!30}2 &\cellcolor{gray!30} 3  \\
					\hhline{=::=::=::=::=:}
					
					\setlength{\fboxrule}{5pt}\fcolorbox{white}{white}{0} & \setlength{\fboxrule}{5pt}\fcolorbox{white}{white}{$1+UV$} & \setlength{\fboxrule}{5pt}\fcolorbox{white}{white}{$U$} & \setlength{\fboxrule}{5pt}\fcolorbox{white}{white}{0} & \setlength{\fboxrule}{5pt}\fcolorbox{white}{white}{$V$}\\
					
					\setlength{\fboxrule}{5pt}\fcolorbox{gray!25}{gray!25}{1} & \setlength{\fboxrule}{5pt}\fcolorbox{gray!25}{gray!25}{2} & \setlength{\fboxrule}{5pt}\fcolorbox{gray!25}{gray!25}{0} & \setlength{\fboxrule}{5pt}\fcolorbox{gray!25}{gray!25}{0} & \setlength{\fboxrule}{5pt}\fcolorbox{gray!25}{gray!25}{0} \\
					
					\setlength{\fboxrule}{5pt}\fcolorbox{white}{white}{2} & \setlength{\fboxrule}{5pt}\fcolorbox{white}{white}{3} & \setlength{\fboxrule}{5pt}\fcolorbox{white}{white}{0} & \setlength{\fboxrule}{5pt}\fcolorbox{white}{white}{1} & \setlength{\fboxrule}{5pt}\fcolorbox{white}{white}{0} \\
					
					\setlength{\fboxrule}{5pt}\fcolorbox{gray!25}{gray!25}{3} & \setlength{\fboxrule}{5pt}\fcolorbox{gray!25}{gray!25}{2} & \setlength{\fboxrule}{5pt}\fcolorbox{gray!25}{gray!25}{0} & \setlength{\fboxrule}{5pt}\fcolorbox{gray!25}{gray!25}{0} & \setlength{\fboxrule}{5pt}\fcolorbox{gray!25}{gray!25}{0} \\
					
				\end{tabular}
				\renewcommand{\arraystretch}{1.1}
			}
		\end{varwidth}
	\vspace{2mm}
		\captionof{table}{$F_{E_4,k,j}(U,V)$}\label{Order4E}
	\end{center}
\end{table}
	\end{minipage} 
\end{figure}
\vspace{-8mm}
\begin{figure}[H]
	\centering 
	\begin{minipage}[b][5.5cm][b]{.40\textwidth}
		\centering
		\begin{table}[H]
	\begin{center}
		\begin{varwidth}{\textheight}\rowcolors{2}{gray!25}{white}
			\resizebox{0.25\textheight}{!}{
				\begin{tabular}{c||c||c||c}
					\hhline{~|~|~|~}
					\backslashbox{$k$}{$j$}  &
					\cellcolor{gray!30}0 & \cellcolor{gray!30}1 & \cellcolor{gray!30}2   \\
					\hhline{=::=::=::=:}
					\setlength{\fboxrule}{5pt}\fcolorbox{white}{white}{0} & \setlength{\fboxrule}{5pt}\fcolorbox{white}{white}{$1+UV$} & \setlength{\fboxrule}{5pt}\fcolorbox{white}{white}{$U$} & \setlength{\fboxrule}{5pt}\fcolorbox{white}{white}{$V$}\\
					\setlength{\fboxrule}{5pt}\fcolorbox{gray!25}{gray!25}{1} & \setlength{\fboxrule}{5pt}\fcolorbox{gray!25}{gray!25}{3} & \setlength{\fboxrule}{5pt}\fcolorbox{gray!25}{gray!25}{0} & \setlength{\fboxrule}{5pt}\fcolorbox{gray!25}{gray!25}{0}  \\
					\setlength{\fboxrule}{5pt}\fcolorbox{white}{white}{2} & \setlength{\fboxrule}{5pt}\fcolorbox{white}{white}{3} & \setlength{\fboxrule}{5pt}\fcolorbox{white}{white}{0} & \setlength{\fboxrule}{5pt}\fcolorbox{white}{white}{0} \\
				\end{tabular}
			}
			
		\end{varwidth}
	\vspace{2mm}
		\captionof{table}{$F_{E_3,k,j}(U,V)$}\label{Order3E}
	\end{center}
\end{table}
	\end{minipage}
	\quad \quad \quad \quad \quad
	\begin{minipage}[b][5cm][b]{.40\textwidth}
		\centering
		\begin{table}[H]
	\begin{center}
		\begin{varwidth}{\textheight}\rowcolors{2}{gray!25}{white}
			\resizebox{0.25\textheight}{!}{
				\begin{tabular}{c||c||c}
					\hhline{~|~|~}
					\backslashbox{$k$}{$j$}  &
					\cellcolor{gray!30}0 & \cellcolor{gray!30}1    \\
					\hhline{=::=::=:}
					\setlength{\fboxrule}{5pt}\fcolorbox{white}{white}{0} & \setlength{\fboxrule}{5pt}\fcolorbox{white}{white}{$1+UV$} & \setlength{\fboxrule}{5pt}\fcolorbox{white}{white}{$U+V$} \\
					\setlength{\fboxrule}{5pt}\fcolorbox{gray!25}{gray!25}{1} & \setlength{\fboxrule}{5pt}\fcolorbox{gray!25}{gray!25}{4} & \setlength{\fboxrule}{5pt}\fcolorbox{gray!25}{gray!25}{0} \\
			\end{tabular}}	
		\end{varwidth}
	\vspace{2mm}
		\captionof{table}{$F_{E_2,k,j}(U,V)$}\label{Order2E}
	\end{center}
\end{table}
	\end{minipage} 
\end{figure}

Therefore by Theorem \ref{wzorek} the number $h^{p,q}(X_{d,n})$ equals:

\begin{equation*}
		\resizebox{\linewidth}{!}{ 
			$\begin{cases}
			\left\{(U+V)^{n}+\left(UV+4\sqrt{UV}+1\right)^{n}\right\}[U^{p}V^{q}] & \textup{if } d=2, \\
			
			\left\{U^{n}+V^{n}+\left(1+\sqrt[3]{UV}\right)^{3n}\right\}[U^{p}V^{q}] & \textup{if } d=3, \\
			
			\begin{aligned}[b]
				&\Bigg\{U^{n}+V^{n}+\left(1+UV+2\sqrt[4]{UV}+3\sqrt[4]{(UV)^2}+2\sqrt[4]{(UV)^3}\right)^{n}+\left(\sqrt[4]{(UV)^2}\right)^{n}
				\Bigg\}[U^{p}V^{q}]
			\end{aligned}& \textup{if } d=4,\\
			
			\begin{aligned}[b]&\Bigg\{U^{n}+V^{n}+\left(1+UV+\sqrt[6]{UV}+2\sqrt[6]{(UV)^{2}}+2\sqrt[6]{(UV)^{3}}+2\sqrt[6]{(UV)^{4}} +\sqrt[6]{(UV)^{5}}\right)^{n}+\\&+2\cdot (UV)^{\frac{n}{2}}+ \left(\sqrt[6]{(UV)^2}+\sqrt[6]{(UV)^4}\right)^{n}\Bigg\}[U^{p}V^{q}]\end{aligned}&\textup{if } d=6.
		\end{cases}$}
	\end{equation*} 

\subsection{Hodge numbers of $Y_{d,n}$}

In this section we shall derive formulas for Hodge numbers of manifolds $Y_{d,n}.$ We shall treat separately cases $2,$ $3,$ $4$ and $6.$ 
\subsubsection{$Y_{6,n}$}
In the formulas for Hodge numbers of $Y_{6,n}$ several invariants appear which describe $\Fix(\gamma_{d}^{k})$ and were introduced in \cite{C}:\label{notation}

\begin{itemize}\itemsep=1mm\leftskip=-10mm
\item[] $S_{6}$ -- $K3$ surfaces with a non-symplectic automorphism $\gamma_{6}\colon S_{6}\to S_{6}$ of order $6,$
\item[] $E_{6}$ -- elliptic curve with the Weierstrass equation $y^{2}=x^3+1,$ and automorphism $\alpha_{6}(x,y) = (\zeta_{6}^2x,-y),$ where $\zeta_{6}$ denotes a fixed $6$-th root of unity satisfying $\zeta_6^2=\zeta_3,$ 
\item[] $r=\dim H^{2}(S_6,\mathbb{C})^{\gamma_{6}},$ 
\item[] $m=\dim H^{2}(S_6,\mathbb{C})_{\zeta_{6}^{i}}$ for $i\in\{1,5\},$
\item[] $\alpha=\dim H^{2}(S_6,\mathbb{C})_{\zeta_{6}^{i}}$ for $i\in\{2,4\},$
\item[] $\beta=\dim H^{2}(S_6,\mathbb{C})_{\zeta_{6}^{3}}$,
\item[] $\Fix(\alpha_{6})=\Fix(\alpha_{6}^{3})=\{f_{1}\},$ $\Fix(\alpha_{6}^{2})=\{f_{1},f_{2},f_{3}\},$ where $\alpha_{6}(f_{2})=f_{3}$, $\alpha_{6}(f_{3})=f_{2}$ and $\Fix(\alpha_{6}^{3})=\{f_{1},f_{4},f_{5},f_{6}\},$ where $\alpha_{6}(f_{4})=f_{5}$, $\alpha_{6}(f_{5})=f_{6}$ and $\alpha_{6}(f_{6})=f_{4},$
\item[] $\Fix\left(\gamma_{6} \right)=\{K_{1}, \ldots K_{\ell-1}\}\cup \{D\}\cup \{P_{1}, \ldots, P_{p_{(2,5)}}\}\cup
	\{Q_{1}, \ldots, Q_{p_{(3,4)}}\},$ where
	\begin{itemize}\itemsep=1mm
		\item the set $\{K_{1}, \ldots K_{\ell-1}\}\cup \{D\}$ consists of curves which are fixed by $\gamma_{6}$ together with the curve $D$ of maximal genus $g(D),$ in fact $K_{i}$ are rational,
		\item $\{P_{1},\ldots, P_{p_{(2,5)}}\}$ is the set of points such that linearisation of $\gamma_{6}$ near the fixed point is represented by the diagonal matrix $\diag(\zeta_{6}^{2}, \zeta_{6}^{5}),$
		\item $\{Q_{1}, \ldots, Q_{p_{(3,4)}}\}$ is the set of points such that linearisation of $\gamma_{6}$ near the fixed point is represented by the diagonal matrix $\diag(\zeta_{6}^{3}, \zeta_{6}^{4}),$
	\end{itemize}
	\item[] $\lam{\Fix\left(\gamma_{6}^{2}\right)=\{L_{1},\ldots ,L_{k-2b-1}\}\cup \{G\}\cup \{(A_{1}, A_{1}'), \ldots, (A_{b}, A_{b}')\}\cup \{(R_{1}, R_{1}'),\ldots, (R_{n'}, R_{n'}')\}\cup \{P_{1}, \ldots, P_{p_{(2,5)}}\}},$ where
	\begin{itemize}\itemsep=1mm
		\item the set $\{L_{1},\ldots L_{k-2b-1}\}\cup \{G\}$ consists of curves which are fixed by $\gamma_{6}^{2}$ together with the curve $G$ of maximal genus $g(G),$ in fact $L_{i}$ are rational,
		\item $\{(A_{1}, A_{1}'),\ldots, (A_{b}, A_{b}')\}$ is the set of all pairs $(A_i, A_i')$ of curves which are fixed by $\gamma_{6}^{2}$ and $\gamma_{6}(A_i)=A_i'$ (curves of the third type),
		\item $\{(R_{1}, R_{1}'),\ldots, (R_{n'}, R_{n'}')\}$ is the set of pairs of points fixed by $\gamma_{6}^{2}$ and such that $\gamma_{6} (R_{i})=R_{i}',$
	\end{itemize}
	\item[] $\Fix\left(\gamma_{6}^{3}\right)=\{(M_{1}, M_{1}',M_{1}''),\ldots, (M_{a}, M_{a}',M_{a}'')\}\cup \{T_{1}, \ldots, T_{N-3a-2}\}\cup \{F_{1}, F_{2}\},$ where
	\begin{itemize}\itemsep=1mm
		\item the set $\{(M_{1}, M_{1}',M_{1}''),\ldots, (M_{a}, M_{a}',M_{a}'')\}$ consists of curves which are fixed by $\gamma_{6}^{3}$ and such that $\gamma_{6}(M_{i})=M_{i}',$ $\gamma_{6}(M_{i}')=M_{i}''$ and $\gamma_{6}(M_{i}'')=M_{i},$
		\item the set $\{T_{1}, \ldots, T_{N-3a-2}\}\cup \{F_{1}, F_{2}\}$ consists of curves fixed by $\gamma_{6}$ (and so $\gamma_{6}^{3}$), together with curves $F_{1}$ and $F_{2}$ of maximal genera $g(F_{1})$ and $g(F_{2}),$ respectively.
	\end{itemize}
\end{itemize}
The Poincar{\'e} polynomials $F_{S_6,k,j}(U,V)$ equal:
\begin{equation*}
		\begin{cases}
			(UV)^2+r\cdot UV+1 & \textup{if } k=0,\; j=0,\\
			U^2+(m-1)\cdot UV & \textup{if } k=0,\; j=1,\\
			\alpha \cdot UV & \textup{if } k=0,\; j=2,4,\\ 
			\beta  \cdot UV & \textup{if } k=0,\; j=3,\\
			V^2+(m-1)\cdot UV & \textup{if } k=0,\; j=5,\\
			\ell+p_{(2,5)}\cdot UV+p_{(3,4)}\cdot UV+g(D)\cdot(U+V)+\ell\cdot  UV & \textup{if } k=1,\; j=0, \\
			
			k-b+n{'}\cdot UV+p_{(2,5)}\cdot UV+ UV+g+(G/\gamma_{6})+(k-b)\cdot UV & \textup{if } k=2,\; j=0,\\
			
			b+n{'}\cdot UV+\Big(g(G)-g(G/\gamma_{6})\Big)\cdot(U+V)+b\cdot UV &\textup{if } k=2,\; j=3,\\
			
			N-2a+\left(g(F_1/\gamma_{6})+g(F_2/\gamma_{6})\right)\cdot(U+V)+(N-2a)\cdot UV &\textup{if } k=3,\; j=1,\\
			
			a+\frac{1}{2}\Big(g(F_{1})+g(F_{2})-g(F_1/\gamma_{6})-g(F_2/\gamma_{6})\Big)(U+V)+a\cdot UV&\textup{if } k=3,\; j=2,4,\\
			
			k-b+n{'}+p_{(2,5)}+g(G/\gamma_{6})(U+V)+(k-b)\cdot UV &\textup{if } k=4,\; j=0,\\
			
			b+n{'}+\left(g(G)-g(G/\gamma_{6})\right)\cdot(U+V)+b\cdot UV&\textup{if } k=4,\; j=3,\\
			0&\textup{otherwise.} 			
			
		\end{cases} \end{equation*}
Therefore by Theorem \ref{wzorek} the Poincar{\'e} polynomial of $Y_{6,n}$ equals:
\bgroup\allowdisplaybreaks
\begin{align*}
	&\sum\limits_{j=0}^{5}\Bigg(F_{S_6,0,j}+\sqrt[6]{UV}F_{S_6,1,j}+\sqrt[6]{(UV)^{2}}F_{S_6,2,j}+\sqrt[6]{(UV)^{3}}F_{S_6,3,j}+\sqrt[6]{(UV)^{4}}F_{S_6,4,j}+\sqrt[6]{(UV)^{5}}F_{S_6,5,j}\Bigg)\times\\&\times\Bigg(F_{E_6,0,j}+\sqrt[6]{UV}F_{E_6,1,j}+\sqrt[6]{(UV)^{2}}F_{E_6,2,j}+\sqrt[6]{(UV)^{3}}F_{E_6,3,j}+\sqrt[6]{(UV)^{4}}F_{E_6,4,j}+\sqrt[6]{(UV)^{5}}F_{E_6,5,j}\Bigg)^{n-1}=\\&=\Bigg((UV)^2+r\cdot UV+1+\sqrt[6]{UV}\cdot \Big(\ell+p_{(2,5)}\cdot UV+p_{(3,4)}\cdot UV+g(D)\cdot(U+V)+\ell\cdot  UV\Big)+\\&+\sqrt[6]{(UV)^2}\cdot\Big(k-b+n'\cdot UV+p_{(2,5)}\cdot UV+g\left(G/\phi_{6}^{S_6}\right)(U+V)+(k-b)\cdot UV \Big)+\\&+\sqrt[6]{(UV)^3}\cdot \Big(N-2a+\left(g\left(\quot{F_1}{\gamma_{6}}\right)+g\left(\quot{F_2}{\gamma_{6}}\right)\right)\cdot(U+V)+(N-2a)\cdot UV \Big)+\\&+\sqrt[6]{(UV)^4}\cdot\Big(k-b+n'+p_{(2,5)}+g(G/\gamma_{6})(U+V)+(k-b)\cdot UV \Big)+\\&+\sqrt[6]{(UV)^5}\cdot\Big(\ell+p_{(2,5)}+p_{(3,4)}+g(D)\cdot(U+V)+\ell\cdot UV \Big)\Bigg)\cdot\\&\cdot \Big(1+UV+\sqrt[6]{UV}+2\sqrt[6]{(UV)^2}+2\sqrt[6]{(UV)^3}+2\sqrt[6]{(UV)^4}+\sqrt[6]{(UV)^5}\Big)^{n-1}+\\&+\Big(U^2+(m-1)\cdot UV \Big)\cdot U^{n-1}+\Bigg(\alpha\cdot UV+\sqrt[6]{(UV)^3}\cdot\Big(a+\frac{1}{2}\Big(g(F_{1})+g(F_{2})-\\&-g\left(\quot{F_1}{\gamma_{6}}\right)-g\left(\quot{F_2}{\gamma_{6}}\right)\Big)\cdot(U+V)+a\cdot UV\Big) \Bigg)\cdot \Big(\sqrt[6]{(UV)^3} \Big)^{n-1}+\\&+
	\Bigg(\beta\cdot UV+\sqrt[6]{(UV)^2}\cdot\Big(b+n'\cdot UV+\Big(g(G)-g(G/\gamma_{6})\Big)\cdot(U+V)+b\cdot UV\Big)+\\&+\sqrt[6]{(UV)^4}\cdot\Big(b+n'+\Big(g(G)-g(G/\gamma_{6})\Big)\cdot(U+V)+b\cdot UV\Big)\Bigg)\cdot \Big(\sqrt[6]{(UV)^2}+\sqrt[6]{(UV)^4} \Big)^{n-1}+\\&+ \Bigg(\alpha\cdot UV+\sqrt[6]{(UV)^3}\cdot\Big(a+\frac{1}{2}\Big(g(F_{1})+g(F_{2})-g\left(\quot{F_1}{\gamma_{6}}\right)-g\left(\quot{F_2}{\gamma_{6}}\right)\Big)\cdot(U+V)+\\&+a\cdot UV\Big) \Bigg)\cdot \Big(\sqrt[6]{(UV)^3} \Big)^{n-1}+\Big(Y^{2}+(m-1)\cdot UV \Big)\cdot V^{n-1}.
\end{align*}\egroup

%From this formula we can compute the Euler characteristic by evaluating the above formula at $6$-th roots of unity. %Therefore we can compute this as a sum of six geometric sequences, hence there exists a recurrence of degree at most 6. In fact two of them coincide and so we obtain a recurrence of order 4.

\subsubsection{$Y_{4,n}$}

Let us keep the following notation for invariants of $K3$ surface and elliptic curve given in \cite{C}:

\begin{itemize}\itemsep=1mm\leftskip=-10mm
		\item[] $S_{4}$ -- $K3$ surfaces with a non-symplectic automorphism $\gamma_{4}\colon S_{4}\to S_{4}$ of order $4$
		\item[] $E_{4}$ -- elliptic curve with the Weierstrass equation $y^{2}=x^3+x,$ and automorphism $\alpha_{4}$ is given by $\alpha_{4}(x,y) = (-x,iy),$
		\item[]$r=\dim H^{2}(S_4,\mathbb{C})^{\gamma_{4}},$ 
		\item[]$m=\dim H^{2}(S_4,\mathbb{C})_{\zeta_{4}^{i}}$ for $i\in\{1,2\},$
		\item[] $\alpha=\dim H^{2}(S_4,\mathbb{C})_{\zeta_{4}^{2}}$,
	\item[] $\Fix(\alpha_{4})=\Fix(\alpha_{4}^{3})=\{f_{1},f_{2}\},$ $\Fix(\alpha_{4}^{2})=\{f_{1},f_{2},f_{3},f_{4}\},$ where $\alpha_{4}(f_{3})=f_{4}$ and $\alpha_{4}(f_{4})=f_{3},$ 
	\item[]$\Fix\left(\gamma_{4}^{2}\right)=L_{1}\cup L_{2} \cup \ldots \cup L_{N-b-2a-1}\cup \{D\}\cup \{(A_{1}, A_{1}'), (A_{2}, A_{2}'),\ldots, (A_{a}, A_{a}')\}\cup \{B_{1}, B_{2}, \ldots, B_{b}\},$ where
	\begin{itemize}\itemsep=1mm
		\item $\{(A_{1}, A_{1}'), (A_{2}, A_{2}'),\ldots, (A_{a}, A_{a}')\}$ is the set of all pairs $(A_i, A'_i)$ of curves which are fixed by $\gamma_{4}^{2}$ and $\gamma_{4}(A_i)=A'_i$,
		\item $\{B_{1}, B_{2}, \ldots, B_{b}\}$ is the set of curves which are fixed by $\gamma_{4}^{2}$ and are non-invariant by $\gamma_{4}$.
	\end{itemize}
	\item[] $\Fix\left(\gamma_{4}\right)=\{R_{1}, R_{2},\ldots R_{k-1}\}\cup \{G\}\cup \{P_{1}, P_{2}, \ldots, P_{n_{1}}\}\cup
	\{Q_{1}, Q_{2}, \ldots, Q_{n_{2}}\},$ where
	\begin{itemize}\itemsep=1mm
		\item the set $\{R_{1}, R_{2},\ldots R_{k-1}\}\cup \{G\}$ contains of curves which are fixed by $\gamma_{4}$ together with the curve $G$ of maximal genus $g(G),$ curves $R_{i}$ are rational,
		\item $\{P_{1}, P_{2}, \ldots, P_{n_{1}}\}$ is the set of points which are fixed by $\gamma_{4}$ not laying on the curve $D,$ 
		\item $\{Q_{1}, Q_{2}, \ldots, Q_{n_{2}}\}$ is the set of points which are fixed by $\gamma_{4}$ laying on the curve $D.$ 
	\end{itemize}
\end{itemize}
The Poincar{\'e} polynomials $F_{S_6,k,j}(U,V)$ equal:
\begin{equation*}
	 \begin{cases}
			(UV)^2+r\cdot UV+1 & \textup{if } k=0,\; j=0,\\
			U^2+(m-1)\cdot UV & \textup{if } k=0,\; j=1,\\
			(22-r-2m) \cdot UV & \textup{if } k=0,\; j=2,\\ 
			V^2+(m-1)\cdot UV & \textup{if } k=0,\; j=3,\\
			k+(n_1+n_2)\cdot UV+g(G)\cdot(U+V)+k\cdot UV & \textup{if } k=1,\; j=0, \\
			
			N-a+g\left(\quot{D}{\gamma_{4}}\right)\cdot(U+V)+(N-a)\cdot UV & \textup{if } k=2,\; j=0,\\
			
			a+\left(g(D)-g\left(\quot{D}{\gamma_{4}}\right)\right)(U+V)+a\cdot UV &\textup{if } k=2,\; j=2,\\
			
			k+n_1+n_2+g(G)\cdot(U+V)+k\cdot UV &\textup{if } k=3,\; j=0,\\
			0&\textup{otherwise.} 			
			
		\end{cases} \end{equation*}
Therefore by Theorem \ref{wzorek} the Poincar{\'e} polynomial of $Y_{4,n}$ equals:
\bgroup\allowdisplaybreaks
\begin{align*}&\Bigg((UV)^2+r\cdot UV+1+\Big(k+(n_1+n_2)\cdot UV+g(G)\cdot(U+V)+k\cdot UV\Big)\cdot \sqrt[4]{UV}+\\&+\Big(N-a+g\left(\quot{D}{\gamma_{4}}\right)\cdot(U+V)+(N-a)\cdot UV\Big)\cdot\sqrt[4]{(UV)^2}+\Big(k+n_1+n_2+\\&+g(G)\cdot(U+V)+k\cdot UV\Big)\cdot\sqrt[4]{(UV)^3}\Bigg)\cdot \Big(1+UV+2\sqrt[4]{UV}+3\sqrt[4]{(UV)^3}+2\sqrt[4]{(UV)^3}\Big)^{n-1}+\\&+\Big(U^2+(m-1)\cdot UV\Big)\cdot U^{n-1}+\Big((22-r-2m) \cdot UV+ \Big(a+\left(g(D)-g\left(\quot{D}{\gamma_{4}}\right)\right)(U+V)+\\&+a\cdot UV\Big)\cdot\sqrt[4]{(UV)^{2}} \Big)\cdot\Big(\sqrt[4]{(UV)^{2}}\Big)^{n-1}+ \Big(V^2+(m-1)\cdot UV\Big)\cdot V^{n-1}. \end{align*}\egroup

\subsubsection{$Y_{3,n}$}	

Let us keep the following notation for invariants of $K3$ surface and elliptic curve given in \cite{C}:

\begin{itemize}\itemsep=1mm\leftskip=-10mm
		\item[] $S_{3}$ -- $K3$ surfaces with a non-symplectic automorphism $\gamma_{3}\colon S_{3}\to S_{3}$ of order $3,$
		\item[] $E_{3}$ -- elliptic curve with the Weierstrass equation $y^{2}=x^3+1,$ and automorphism $\alpha_{3}$ is given by $\alpha_{3}(x,y) = (\zeta_{3}x,y),$
		\item[]$r=\dim H^{2}(S_3,\mathbb{C})^{\gamma_{3}},$ 
		\item[]$m=\dim H^{2}(S_3,\mathbb{C})_{\zeta_{3}},$
		\item[] $\Fix\left(\gamma_{3}\right)=\Fix\left(\gamma_{3}^{2}\right)=\{f_{1},f_{2},f_{3}\},$ 
	    \item[] $\Fix\left(\gamma_{3}\right)=L_{1}\cup L_{2} \cup \ldots \cup L_{k-1}\cup C\cup \{P_{1}, P_{2}, \ldots, P_{h}\},$ where
	\begin{itemize}\itemsep=1mm
		\item the set $\{L_{1}, L_{2},\ldots L_{k-1}\}\cup \{C\}$ consists of curves which are fixed by $\gamma_{3}$ together with the curve $C$ of maximal genus $g(C),$ in fact $L_{i}$ are rational,
		\item $\{P_{1}, P_{2}, \ldots, P_{n}\}$ is the set of points which are fixed by $\gamma_{3}.$
	\end{itemize}
\end{itemize}
The Poincar{\'e} polynomials $F_{S_3,k,j}(U,V)$ equal:

\begin{equation*}
	 \begin{cases}
			(UV)^2+r\cdot UV+1 & \textup{if } k=0,\; j=0,\\
			U^2+(m-1)\cdot UV & \textup{if } k=0,\; j=1,\\
			V^2+(m-1)\cdot UV & \textup{if } k=0,\; j=2,\\
			k+h\cdot UV+g(C)\cdot(U+V)+k\cdot UV & \textup{if } k=1,\; j=0, \\
			
			k+h+g(C)\cdot(U+V)+k\cdot UV & \textup{if } k=2,\; j=0,\\
			0&\textup{otherwise.} 			
			
		\end{cases} \end{equation*}
Therefore by Theorem \ref{wzorek} the Poincar{\'e} polynomial of $Y_{3,n}$ equals:
\bgroup\allowdisplaybreaks
\begin{align*}&\Bigg((UV)^2+r\cdot UV+1+\Big(k+h\cdot UV+g(C)\cdot(U+V)+k\cdot UV\Big)\cdot \sqrt[3]{UV}+\Big(k+h+\\&+g(C)\cdot(U+V)+k\cdot UV\Big)\cdot\sqrt[3]{(UV)^2}\Bigg)\cdot \Big(1+UV+3\sqrt[3]{UV}+3\sqrt[3]{(UV)^2}\Big)^{n-1}+\\&+\Big(U^2+(m-1)\cdot UV\Big)\cdot U^{n-1}+\Big(V^2+(m-1)\cdot UV\Big)\cdot V^{n-1}. \end{align*}\egroup

\subsubsection{$Y_{2,n}$}
\label{rzad2}

For a $K3$ surface $S_2$ with involution we have the following invariants:

\begin{itemize}\itemsep=1mm\leftskip=-10mm
	\item[]$S_{2}$ -- $K3$ surfaces with a non-symplectic automorphism $\gamma_{2}\colon S_{2}\to S_{2}$ involution,
	\item[]$E_{2}$ -- arbitrary elliptic curve with involution $\alpha_{2}(x,y) = (x,-y),$
	\item[]$r=\dim H^{2}(S_2,\mathbb{C})^{\gamma_{2}},$
	\item[]$m=\dim H^{2}(S_2,\mathbb{C})_{\zeta_{2}^{}},$
	\item[]$\Fix\left(\alpha_{2}\right)=\{a,b,c,d\},$
	\item[]$\Fix\left(\gamma_{2}\right)=C_{1}\cup C_{2} \cup \ldots \cup C_{N}$ where
	the set $\{C_{1}, C_{2},\ldots ,C_{N}\}$ contains of curves which are fixed by $\gamma_{2}$ with sum of genera equals $N'.$
\end{itemize}
The Poincar{\'e} polynomials $F_{S_2,k,j}(U,V)$ equal:
\begin{equation*}
	 \begin{cases}
			(UV)^2+r\cdot UV+1 & \textup{if } k=0,\; j=0,\\
			U^2+V^2+(m-2)\cdot UV & \textup{if } k=0,\; j=1,\\
			N+N'\cdot(U+V)+N\cdot UV & \textup{if } k=1,\; j=0,\\
			0 &\textup{if } k=1,\; j=1. 	
		\end{cases} \end{equation*}
Therefore by Theorem \ref{wzorek} the Poincar{\'e} polynomial of $Y_{2,n}$ equals:	
\bgroup\allowdisplaybreaks
\begin{align*} &\Bigg((UV)^2+r\cdot UV+1+\Big(N+N'\cdot(U+V)+N\cdot UV\Big)\cdot \sqrt{UV}\Bigg)\cdot \Big(1+UV+4\sqrt{UV}\Big)^{n-1}+\\&+\Big(U^2+V^2+(m-2)\cdot UV\Big)\cdot (U+V)^{n-1}. \end{align*}\egroup
Using the following relations (see \cite{V}): $$r=10+N-N'\quad \textup{and} \quad m=12-N+N'$$ we can rewrite the above formula in terms of $N$ and $N',$ i.e. 
\begin{align*}&\Bigg((UV)^2+(10+N-N')\cdot UV+1+\Big(N+N'\cdot(U+V)+N\cdot UV\Big)\cdot \sqrt{UV}\Bigg)\times\\ &\times \Big(1+UV+4\sqrt{UV}\Big)^{n-1}+\Big(U^2+V^2+(10-N+N')\cdot UV\Big)\cdot (U+V)^{n-1}. \end{align*}

\begin{Rem} We created supporting file containing \texttt{MAPLE} procedures computing Poincar{\'e} polynomials of manifolds $X_{d,n}$ and $Y_{d,n}$ (\cite{Kody}). \end{Rem}

\section{zeta functions of a finite quotient}

In this section we prove a formula which is analogue to the one stated in Thm. \ref{wzorek} for zeta functions of a finite quotients. Since the zeta function may be different for any member in the family and as it is an arithmetic function depending on prime power $q,$ the problem is much more subtle.  

Let $V,$ $W$ be finite dimensional vector spaces over a field $\mathbb{K}$ and take $L\in \textup{End}(V),$ $M\in \textup{End}(W).$ We define their characteristic polynomials:
$$f(t):=\det(\mathbbm{1}-t\cdot L)\quad \textup{and}\quad g(t):=\det(\mathbbm{1}-t\cdot M). $$ Over an algebraic closure $\bar{\mathbb{K}}$ of $\mathbb{K}$ we have factorisations $$f(t)=(1-\lambda_{1}t)(1-\lambda_{2}t)\ldots (1-\lambda_{\dim V}t)$$ and 
$$g(t)=(1-\mu_{1}t)(1-\mu_{2}t)\ldots (1-\mu_{\dim W}t),$$ for some $\lambda_{1}, \lambda_{2}, \ldots, \lambda_{\dim V},\mu_{1}, \mu_{2}, \ldots, \mu_{\dim W}\in \bar{\mathbb{K}}$ which are in fact eigenvalues of endomorphisms $L_{\bar{\mathbb{K}}}$~and $M_{\bar{\mathbb{K}}}$ respectively.   

Denoting by $f\otimes g$ the characteristic polynomial of $L\otimes M$: $$f\otimes g(t)=\det(\mathbbm{1}-t\cdot L\otimes M),$$ can be written as:
$$f\otimes g(t):=\prod_{i=1}^{\dim V}\prod_{j=1}^{\dim W}(1-\lambda_{i}\mu_{j}t).$$ One can see that $$f\otimes g(t)=\prod_{i=1}^{\dim V}g(\lambda_{i}t)=\prod_{j=1}^{\dim W}f(\mu_{j}t).$$ Moreover one can check that this polynomial can be computed using the resultant i.e. $$f\otimes g(t)= \textrm{res}_{s}\left(f(s), s^{\deg(g)}\cdot g\left( \frac{t}{s}\right) \right).$$

The tensor product of polynomials extends uniquely to the case of rational functions~$f=\dfrac{a}{b},$ $g=\dfrac{c}{d}$, where $a, b, c, d\in \mathbb{K}[t]$ by taking:

$$\frac{a}{b}\otimes \frac{c}{d}=\frac{(a\otimes c)\cdot (b\otimes d)}{(a\otimes d)\cdot (b\otimes c)}.$$

Rose in \cite{Rose} introduced the \textit{orbifold Frobenius morphisms} on the Chen-Ruan orbifold cohomology and used it to define the \textit{orbifold zeta function} %If the coarse moduli stack $\textit{X}$ admits a crepant resolution of singularities $\tylda{\mathcal{X}}$ then the orbifold zeta function coincides with zeta function of $\tylda{\mathcal{X}}$ %Deligne-Mumford stack $\mathcal{X}$ over $\FF_{q}$ to be:
\begin{equation}\label{rosee} Z_{\textrm{CR}}(\mathcal{X}, t):=\det \left(1-\Frob_{\orb}t \mid H^{*}_{\textrm{CR}}\left(\mathcal{X}\times_{\FF_{q}}\bar{\FF_{q}}, \QQ_{l}\right)\right), \end{equation} where $H^{*}_{\textrm{CR}}(\mathcal{X}, \QQ_{l})$ is $\ell$-adic Chen Ruan cohomology and $\Frob_{\orb}$ is the \textit{orbifold Frobenius morphism} defined in Section 3 and 5, respectively of \cite{Rose}. It was also proven (\cite[Corollary 6.4]{Rose}) that for a crepant resolution $\tylda{X}\to X$ of the course moduli scheme $X$ of $\mathcal{X}$ the orbifold cohomological zeta function of $\mathcal{X}$ coincides with classical zeta function of $\tylda{X}$ i.e. $$Z_{H^{*}_{\textrm{CR}}}(\mathcal{X}, t)=Z_{q}(\tylda{X},t).$$

\begin{Thm}\label{zetawzorek} Let $X_{i}$ be a variety of Calabi-Yau type with a $pnpp$ automorphism $\phi_{i,d}\colon X_i\to X_i$ of order $d$. Suppose that there exists a crepant resolution $\mathcal{X}_{d,n}$ of the quotient variety $$(X_{1}\times X_{2}\times \ldots \times X_{n})/ \ZZ_{d}^{n-1}.$$ Then the zeta function $Z_{q}\left(\mathcal{X}_{d,n}\right)(T)$ equals the product of factors of the following rational function in~$T$ \begin{equation}\left(\prod_{j=0}^{d-1}\bigotimes_{i=1}^{n}\left(\prod_{m=0}^{d-1}Z_{X_{i}, m, j}\left(q^{\frac{m}{d}}T\right)\right)\right)^{(-1)^{n+1}} \end{equation} containing only integral powers of $q,$ where
	%product of factors of rational function which contain only integral powers of q
	%$$Z_{X_{i},m,j}(T):=\prod_{k=0}^{2\dim X_{i}}\det\left(1-\Fr_{q}^{*}T\mid H^{*}\left( \Fix(\phi_{i,d}^m)_{\zeta_{d}^j}\right)\right)^{(-1)^{k+1}}$$ for $i\in \{1,2,\ldots, n\}$ and $m,j\in \{0,1,\ldots, d-1\}.$  
	$$Z_{X_{i},m,j}(T)=\prod_{\lambda_{i}\geq 0}\prod_{0\leq k_{i}\leq 2\dim X_{i}}\det\left(1-\Fr_{q}^{*}T\mid H^{k_{i}}\left( X_{i,m,\lambda_{i}}\right)_{\zeta_{d}^j}\right)^{(-1)^{k_i+1}}(q^{\lambda_i}T).$$
\end{Thm}

\begin{proof}
	Similarly as in the proof of \ref{wzorek} merging formulas \ref{rosee} and \ref{orbfor} we have the following formula for zeta function of $\mathcal{X}_{d,n}$:
	\bgroup\allowdisplaybreaks
	\begin{align*}\label{finalnywzorek} &Z_{q}\left(\mathcal{X}_{d,n}\right)(T)=\prod_{m\in G_{d,n}}\prod_{\mathfrak{\lambda}\geq 0}\prod_{|\mathfrak{k}|=k-2\frac{|\mathfrak{m}|}{d}-\mathfrak{\lambda}}\prod_{j=0}^{d-1}\bigotimes_{i=1}^{n}\det\left(1-\Frob_{q}T\mid H^{k_{i}}\left(X_{i, m_i,\lambda_i}\right)\right)^{(-1)^{k_i+1}}=\\&=
		\prod_{m\in G_{d,n}}\prod_{\mathfrak{\lambda}\geq 0}\prod_{|\mathfrak{k}|=k}\prod_{j=0}^{d-1}\bigotimes_{i=1}^{n}\det\left(1-\Frob_{q}T\mid H^{k_{i}}\left(X_{i, m_i,\lambda_i}\right)\right)^{(-1)^{k_i+1}}\left(q^{\frac{m_i}{d}+\lambda_i}T\right)=\\&=\prod_{m\in G_{d,n}}\prod_{\mathfrak{\lambda}\geq 0}\prod_{|\mathfrak{k}|=k}\prod_{j=0}^{d-1}\Bigg(\bigotimes_{i=1}^{n}\det\left(1-\Frob_{q}T\mid H^{k_{i}}\left(X_{i, m_i,\lambda_i}\right)\right)\Bigg)^{(-1)^{|\mathfrak{k}|+n}}\left(q^{\frac{m_i}{d}+\lambda_i}T\right),
	\end{align*}
	\egroup
	where $\mathfrak{\lambda}=(\lambda_{1}, \lambda_{2}, \ldots, \lambda_{n}),$ $\mathfrak{k}=(k_{1}, k_{2}, \ldots, k_{n})$, $\mathfrak{m}=(m_{1}, m_{2}, \ldots, m_{n})$ and $|\mathfrak{\lambda}|=\lambda_{1} +\lambda_{2}+\ldots+ \lambda_{n},$ $|\mathfrak{k}|=k_{1} +k_{2}+\ldots+ k_{n}$, $|\mathfrak{m}|=m_{1} +m_{2}+\ldots+ m_{n}.$ 
	
	Extending the exterior product $\displaystyle \prod_{m\in G_{n,d}}(\ldots)$ to $\displaystyle \prod_{m\in \ZZ_{d}^{n}}(\ldots)$ we introduce only factors containing fractional powers of $q$. Denote the resulting rational function by %$\tylda{Z_{q}}\left(\mathcal{X}_{d,n}\right)(T)$ is equal to 
	\begin{align*}
		&W(T)=\prod_{m\in \ZZ_{d}^{n}}\prod_{\pmb{\lambda}\geq 0}\prod_{|\mathfrak{k}|=k}\prod_{j=0}^{d-1}\Bigg(\bigotimes_{i=1}^{n}\det\left(1-\Frob_{q}T\mid H^{k_{i}}\left(X_{i, m_i,\lambda_i}\right)\right)\Bigg)^{(-1)^{|\mathfrak{k}|+n}}\left(q^{\frac{m_i}{d}+\lambda_i}T\right)=\\&=\left(\prod_{j=0}^{d-1}\bigotimes_{i=1}^{n}\left(\prod_{m=0}^{d-1}Z_{X_{i, m, j}}\left(q^{\frac{m}{d}}T\right)\right)\right)^{(-1)^{n+1}}.\end{align*}
	
	Let $W=W_{1}^{\alpha_{1}}\cdot W_{2}^{\alpha_{2}}\cdot  \ldots \cdot W_{s}^{\alpha_{s}}$ be the decomposition of $W$ into product of irreducible polynomials $W_{i}\in \ZZ[q^{\frac{1}{d}},T]$, for $\alpha_{i}\in \ZZ.$ Then 
	$$Z_{q}\left(\mathcal{X}_{d,n}\right)(T)=\prod\left\{W_{i}^{\alpha_{1}}\colon W_{i}\in \ZZ[q,T],\; i=1,\ldots,s \right\}.$$
\end{proof}

\subsection{Explicit computation of the zeta function}

\subsubsection{Zeta function of Borcea-Voisin Calabi-Yau threefold}

In the present section we shall use Thm 4.1 to reprove formula for zeta function of classical Borcea-Voisin Calabi-Yau threefolds (from \cite{GOTOYUI}). 

Let $(S, \alpha_{S})$ be a $K3$ surface admitting a non-symplectic involution $\alpha_{S}$. Consider an elliptic curve $E$ with not period preserving involution $\alpha_{E}$. Let us keep notation from the section \ref{rzad2}. Observe that
$$Z_{E,0,0}(T)=\dfrac{1}{(1-T)(1-qT)}, \quad Z_{E,0,1}(T)=1-a_{q}T+qT^{2}, \quad Z_{E,1,1}(T)=1.$$ The polynomial $Z_{E,1,0}$ depends on the number of points in $\Fix(\alpha_{E})$ which are defined over $\FF_q$. Therefore $$Z_{E,1,0}=\begin{cases} \frac{1}{(1-T)^{4}}, & \textup{if all points in $\Fix(\alpha_{E})$ are defined over $\FF_q$,} \\ \frac{1}{(1-T)^{3}(1+T)}, & \textup{if two points in $\Fix(\alpha_{E})$ are defined over $\FF_q$,} \\ \frac{1}{(1-T)^{2}(1+T+T^2)} & \textup{if one point in $\Fix(\alpha_{E})$ are defined over $\FF_q$.} \end{cases}$$
\textbf{The polynomial $Z_{S,0,0}$}:\quad In that case:
$Z_{S,0,0}=Z_{q}\left( H^{**}(S)^{\alpha_{S}}\right).$ The Frobenius map acting on $r$ curves induces permutation $\pi \in S_{r}$ with decomposition into disjoint cycles of lengths $a_{1},$ $a_{2},$ $\ldots,$ $a_{s}$ for some positive integer $s.$ Therefore
$$Z_{S,0,0}=\displaystyle\frac{1}{(1-T)(1-qT)\displaystyle\prod_{i=1}^{s}\big(1-(qT)^{a_{i}}\big)\left(1-q^{2}T\right)}.$$
\textbf{The polynomial $Z_{S,0,1}$}:\quad One can see that $$H^{2}\left(S, \ZZ\right)=T(S)\oplus NS(S),$$ 
where $NS(S)$ is N{\'e}ron-Severi group generated by algebraic cycles on $S$ and $T(S)$ is transcendental lattice of $S$ defined as an orthogonal complement of $NS(S)$ in $H^{2}\left(S, \ZZ\right)$. In general we cannot say much more but in special cases we can find explicit description of the polynomial $Z_{S,0,1}$. In particular if $S$ is one of $K3$ surfaces appearing in Borcea-Voisin construction then the polynomial can be read out from Theorem 5.6 of \cite{GOTOYUI}. \newline \newline 
\textbf{The polynomial $Z_{S,1,0}$}:\quad Let $C_{g}$ be the curve of maximal genus $g$ in $\Fix(\alpha_{S})$. Then we see that
$$Z_{S,1,0}=Z_{q}\left(H_{\et}^{*}\left(\Fix(\alpha_{S})\right)\right)=\frac{\det\left(1-t\cdot \Frob_{q} | H_{\et}^{1}(C_{g}, \QQ_{l})\right)}{(1-T)(1-q^{}T)\displaystyle\prod_{i=1}^{s}\left(1-(qT)^{a_{s}}\right)}.$$ 
\textbf{The polynomial $Z_{S,1,1}$} is obviously equal to 1. 
%Otherwise the local zeta function $Z_{S,1,0}$ is equal to $$\frac{1}{(1-T)(1-q^{}T)\displaystyle\prod_{i=1}^{s}\left(1-(qT)^{a_{s}}\right)}.$$

We shall compute the zeta function of $(S\times E^{n})/G_{d,n}$ in the case when $S$ has particularly nice arithmetic properties. 

Let $S$ be the $K3$ surface with an obvious non-symplectic involution studied in \cite{OnoPeniston}. It can be defined as double cover of $\PP^{2}$ branched along the union of six lines given by $$XYZ(X+\lambda Y)(Y+Z)(Z+X)=0.$$
\begin{comment}
\begin{figure}[H]
	\begin{center}
      \includegraphics[width=0.28\textwidth]{ono.pdf}
	\end{center} 
	%	\captionof{figure}{$E8$}
\end{figure}
\end{comment}
The resolution of the singularities of that surface is obtained in the following way: first we blow up 3 triple points that defines 24 points on the double cover, then we blow up resulting variety at 15 double points. 
Keeping notation from \cite{OnoPeniston} we compute:
$$Z_{S,0,0}(T)=\frac{1}{(1-T)(1-qT)^{19}\left(1-q^{2}T\right)}, \quad Z_{S,1,0}(T)=\frac{1}{(1-T)^{9}(1-q^{}T)^{9}}$$ and $$Z_{S,0,1}(T)=\frac{1}{\left(1-\gamma_{q}qT\right)\left(1-\gamma_{q}\pi^{2}T\right)\left(1-\gamma_{q}\bar{\pi}^{2}T\right)},$$ where $\pi$ and $\bar{\pi}$ are the eigenvalues of the Frobenius on elliptic curve $$E_{\lambda}\colon y^{2}=(x-1)\left(x^2-\frac{1}{\lambda+1} \right)$$ and $\gamma_{q}=\left(\frac{\lambda+1}{p} \right).$

Therefore from Theorem \ref{zetawzorek} we have the following formula for zeta function of the classical Borcea-Voisin threefold:

\bgroup\allowdisplaybreaks
\begin{align*}Z_{q}\Big(\tylda{\quot{S\times E}{\ZZ_{2}}}\Big)&=\Bigg(\Big(Z_{E,0,0}(T)\cdot Z_{E,1,0}\left(\sqrt[]{q}\cdot T\right)  \Big)\otimes \Big(Z_{S,0,0}(T)\cdot Z_{S,1,0}\left(\sqrt[]{q}\cdot T\right)  \Big) \Bigg) \times \\ &\times \Bigg(\Big(Z_{E,0,1}(T)\cdot Z_{E,1,1}\left(\sqrt[]{q}\cdot T\right)  \Big)\otimes \Big(Z_{S,0,1}(T)\cdot Z_{S,1,1}\left(\sqrt[]{q}\cdot T\right) \Big) \Bigg). \\
	&\end{align*}
\egroup

Since we have three cases depending on number of fixed points of $\alpha_{E}$ defined over $\FF_{q},$ resulting local zeta functions are summarized in the following table:

\begin{table}[H]
	\begin{center}
		\begin{varwidth}{\textheight}\rowcolors{2}{gray!25}{white}
			\rowcolors{2}{gray!25}{white}
			\resizebox{0.58\textheight}{!}{
				\begin{tabular}{c||c}
					\hhline{~|~}
					\cellcolor{gray!30}\parbox[c]{0.17\hsize}{\centering\setlength{\fboxrule}{3pt}\fcolorbox{gray!30}{gray!30}{$\displaystyle Z_{E,1,0}$}}	 &\cellcolor{gray!30}\parbox[c]{0.6\hsize}{\centering\setlength{\fboxrule}{3pt}\fcolorbox{gray!30}{gray!30}{$\displaystyle Y_{2,2}=Z_{q}\Big(\tylda{(S\times E)/\ZZ_{2}}\Big)$}}  \\
					\hhline{=::=:}
					\parbox[c]{0.17\hsize}{\centering\setlength{\fboxrule}{1pt}\fcolorbox{white}{white}{$\displaystyle\dfrac{1}{(1-T)^{4}}$}} &   \parbox[c]{0.6\hsize}{\centering\setlength{\fboxrule}{5pt}\fcolorbox{white}{white}{$\displaystyle{\frac {  \left( 1-{\it a_{q}}\,{\it \gamma_{q}}\,qT+{{\it \gamma_{q}}}^{2}{q}^{3}{T}^{2} \right)  \left( 1-{\it a_{q}}\,{{\it \pi}}^{2}{\it \gamma_{q}}\,T+{{\it \pi}}^{4}{{\it \gamma_{q}}}^{2}q{T}^{2} \right)  \left( 1-{\it a_{q}}\,{{\it \bar{\pi}}}^{2
									}{\it \gamma_{q}}\,T+{{\it \bar{\pi}}}^{4}{{\it \gamma_{q}}}^{2}q{T}^{2} \right) }{\left(1-T\right)\left( 1-qT \right) ^{56}  \left( 1-{q}^{2}T \right) ^{56}\left(1-{q}^{3}T \right) }}$}} \\
					\parbox[c]{0.17\hsize}{\centering\setlength{\fboxrule}{1pt}\fcolorbox{gray!25}{gray!25}{$\displaystyle\dfrac{1}{(1-T)^{3}(1+T)}$}}
					
					& \parbox[c]{0.6\hsize}{\centering\setlength{\fboxrule}{3pt}\fcolorbox{gray!25}{gray!25}{$\displaystyle{\frac{ \left( 1-{\it a_{q}}\,{\it \gamma_{q}}\,qT+{{\it \gamma_{q}}}^{2}{q}^{3}{T}^{2} \right)  \left( 1-{\it a_{q}}\,{{\it \pi}}^{2}{\it \gamma_{q}}\,T+{{\it \pi}}^{4}{{\it \gamma_{q}}}^{2}q{T}^{2} \right)  \left( 1-{\it a_{q}}\,{{\it \bar{\pi}}}^{2
									}{\it \gamma_{q}}\,T+{{\it \bar{\pi}}}^{4}{{\it \gamma_{q}}}^{2}q{T}^{2} \right) }{\left( 1-T \right)\left( 1-qT \right) ^{47}  \left( 1+qT \right) ^{9} \left( 1+{q}^{2}T \right) ^{9} \left( 1-{q}^{2}T \right) ^{47}\left( 1-{q}^{3}T \right)  }}$}} \\
					\parbox[c]{0.17\hsize}{\centering\setlength{\fboxrule}{1pt}\fcolorbox{white}{white}{$\displaystyle\dfrac{1}{(1-T)^{2}(1+T+T^2)}$}} &
					\parbox[c]{0.60\hsize}{\centering\setlength{\fboxrule}{3pt}\fcolorbox{white}{white}{$\displaystyle{\frac{\left( 1-{\it a_{q}}\,{\it \gamma_{q}}\,qT+{{\it \gamma_{q}}}^{2}{q}^{3}{T}^{2} \right)  \left( 1-{\it a_{q}}\,{{\it \pi}}^{2}{\it \gamma_{q}}\,T+{{\it \pi}}^{4}{{\it \gamma_{q}}}^{2}q{T}^{2} \right)  \left( 1-{\it a_{q}}\,{{\it \bar{\pi}}}^{2
									}{\it \gamma_{q}}\,T+{{\it \bar{\pi}}}^{4}{{\it \gamma_{q}}}^{2}q{T}^{2} \right) }{\left( 1-T \right)\left( 1-qT \right) ^{38} \left( 1-{q}^{2}T \right) ^{38} \left( 1+{q}^{2}T+{q}^{4}{T}^{2} \right) ^{9}   \left( 1+qT+{q}^{2}{T}^{2} \right) ^{9
									}\left( 1-{q}^{3}T \right)}}$}}
				\end{tabular}
			}
		\end{varwidth}
	\vspace{2mm}
		\captionof{table}{Local zeta functions of $Y_{2,2}$}\label{zetaBV}
	\end{center}
\end{table}

\subsubsection{Zeta function of $Y_{6,n}$}

We shall give details of computation of the zeta function of a Borcea-Voisin Calabi-Yau $(n+1)$-fold with a very particular shape of the Hodge diamond i.e. non-zero Hodge numbers are placed only in the vertices, vertical diagonal and $h^{m-1, m}=h^{m, m-1}=1$, for $m=\frac{n}{2}+1$ when $n$ is even. As the consequence the Hodge structure on the middle cohomology of odd dimensional examples has the following shape $$(1\;0\;\ldots\; 0\; 1\;1\; 0\; \ldots \;0\;1).$$ In particular the transcendental part of the cohomology is of dimension 2 when $n$ is odd, and it is a union of two 2-dimensional pieces when $n$ is even. 

Consider an elliptic curve $E_{6}$ with the Weierstrass equation $y^2=x^3+1$ together with a not period preserving automorphism of order $6$. 

Let $S_6$ be the $K3$ surface no. 18 from Table 1 of \cite{Order6D}. Then $S_6$ is isomorphic to an elliptic $K3$ surface $X\to \PP^{1}$ whose Weierstrass equation is $$\label{rownanie} y^{2}=x^{3}+\lambda (z-1)^{2}z^{5},$$ where $\lambda\in \ZZ$ is a fixed parameter. The surface $S_6$ is equipped with the following $\zeta_{6}$-action $\alpha_{S_{6}}\colon (x,y,t)\to (\zeta_{3}^{2}x, y, z).$
The elliptic fibrations $S_6\to \PP^{1}$ has 1 type $\textbf{IV}$ fiber and $2$ type $\textbf{II}^{*}$ fibers. 
The surface $S_6$ has the following invariants

\begin{table}[H]
	\begin{center}
		\begin{varwidth}{\textheight}%\rowcolors{2}{gray!25}{white}
			%\rowcolors{1}{gray!25}{white}
			\resizebox{0.45\textheight}{!}{
				\begin{tabular}{|c||c||c||c||c||c||c||c||c||c||c||c||c|}
					\cellcolor{gray!30}$r$ & \cellcolor{gray!30}$m$ & \cellcolor{gray!30}$n$ & \cellcolor{gray!30}$n'$ & \cellcolor{gray!30}$k$ &\cellcolor{gray!30}$a$ & \cellcolor{gray!30}$p_{3,4}$ &  \cellcolor{gray!30}$p_{2,5}$ & \cellcolor{gray!30}$\ell$ & \cellcolor{gray!30}$N$ & \cellcolor{gray!30}$b$ & \cellcolor{gray!30}$\alpha$ & \cellcolor{gray!30}$\beta$\\
					\hhline{=::=::=::=::=::=::=::=::=::=::=::=::=:}
					19 & 1 & 9 & 0 & 6 & 0 & 6 & 9 & 3& 10 & 0 & 0 & 1\\
			\end{tabular}}	
		\end{varwidth}
	\end{center}
\end{table}

After a careful study of the resolution of singularities of the surface $S_6$ we get:
\begin{align*} &Z_{S_{6},0,0}=\dfrac{1}{(1-T)(1-qT)^{19}(1-q^2T)}, \quad Z_{S_{6},0,1}=\dfrac{1}{1-\beta_{q}T}, \quad Z_{S_{6},0,3}=\dfrac{1}{1-c_{q}qT}, \quad Z_{S_{6},0,5}=\dfrac{1}{1-\bar{\beta_{q}}T}\\
& Z_{S_{6},1,0}=\dfrac{1}{(1-T)^{3}(1-qT)^{18}}, \quad Z_{S_{6},2,0}=\dfrac{1}{(1-T)^{6}(1-qT)^{15}}, \quad Z_{S_{6},3,0}=\dfrac{1}{(1-T)^{10}(1-qT)^{10}}\\
& Z_{S_{6},4,0}=\dfrac{1}{(1-T)^{16}(1-qT)^{6}}, \quad Z_{S_{6},2,0}=\dfrac{1}{(1-T)^{18}(1-qT)^{3}}, \quad Z_{S_{6},3,2}=1-\delta_{q}T, \quad Z_{S_{6},3,2}=1-\bar{\delta_{q}}T, \\ &Z_{E_{6},0,0}=\dfrac{1}{(1-T)(1-qT)^{}}, \quad Z_{E_{6},0,1}=1-\alpha_{q}T, \quad Z_{E_{6},0,5}=1-\bar{\alpha_{q}}T,\quad Z_{E_{6},2,0}=Z_{E_{6},3,0}=Z_{E_{6},4,0}=\dfrac{1}{(1-T)^{2}},\\
& Z_{E_{6},1,0}=Z_{E_{6},5,0}=Z_{E_{6},2,3}=Z_{E_{6},3,2}=Z_{E_{6},3,4}=Z_{E_{6},4,3}=\dfrac{1}{(1-T)^{}}, 
\end{align*}
where $\alpha_{q}$ $\&$ $\bar{\alpha_{q}}$ (resp. $\beta_{q}$ $\&$ $\bar{\beta_{q}}$, $\delta_{q}$ $\&$ $\bar{\delta_{q}}$) are traces of the Frobenius map $\Frob_{q}$ on $H^{1}(E_{6})$ (resp. $T(S_{6})$ $\&$ $H^{1}(F)$ for a unique elliptic curve $F$ in $S_{6}$ fixed by $\alpha_{S_{6}}$). The curve $F$ is in fact a twist of $E_{6}$ by a character depending on $\lambda$. Finally, $c_{q}=\pm 1$ depending on properties of the unique singular fiber of type $\textbf{IV}$. Namely, this fiber is a union of three lines, one of them is fixed and remaining two are swapped by $\alpha_{S_{6}}.$ Difference of these two lines generates one dimensional subspace $H^{1,1}(S_{6})_{\zeta_{6}^3}.$ Then $c_{q}=+1$ if these two curves are defined over $\FF_{q}$ and $c_{q}=-1$ in the opposite direction. The explicit value of $\alpha_{q}$,  $\bar{\alpha_{q}}$, $\beta_{q}$, $\bar{\beta_{q}}$, $\delta_{q},$ $\bar{\delta_{q}}$ and $c_{q}$ depends on parameter $\lambda.$ Moreover $Z_{S_{6},k,j}=Z_{E_{6},k,j}=1$ for all remaining cases. Applying Theorem \ref{zetawzorek} we get:
\bgroup\allowdisplaybreaks
\begin{align*}&Z_{q}\left(\tylda{(S_{6}\times E_{6})/\ZZ_{6}}\right)=\frac{(1-\alpha_{q}\beta_{q}T)(1-\delta_{q}T)(1-\bar{\delta_{q}}T)(1-\bar{\alpha_{q}}\bar{\beta_{q}}T)}{(1-T)(1-qT)^{103}(1-q^2T)^{103}(1-q^3T)}. \end{align*}
\egroup

Using the same method we are able to compute zeta functions of higher dimensional quotients $\tylda{(S_{6}\times E_{6}^{n-1})/\ZZ_{6}^{n-1}}.$ For higher value of $n$ computations are much more involved but can be easily carried out with computer algebra system. 

In the table below we collect results for $n=2,3,4,$ for $n\geq 5$ formulas are too long to display. 

\begin{table}[H]
	\begin{center}
		\begin{varwidth}{\textheight}\rowcolors{2}{gray!25}{white}
			\rowcolors{2}{gray!25}{white}
			\resizebox{0.58\textheight}{!}{
				\begin{tabular}{c||c}
					\hhline{~|~}
					\cellcolor{gray!30}\parbox[c]{0.05\hsize}{\centering\setlength{\fboxrule}{3pt}\fcolorbox{gray!30}{gray!30}{$n$}}	 &\cellcolor{gray!30}\parbox[c]{0.7\hsize}{\centering\setlength{\fboxrule}{3pt}\fcolorbox{gray!30}{gray!30}{$\displaystyle Z_q(Y_{6,n})=Z_{q}\Big(\tylda{(S_{6}\times E_{6}^{n-1})/\ZZ_{6}^{n-1}}\Big)$}}  \\
					\hhline{=::=:}
					\parbox[c]{0.05\hsize}{\centering\setlength{\fboxrule}{1pt}\fcolorbox{white}{white}{$2$}} &   \parbox[c]{0.7\hsize}{\centering\setlength{\fboxrule}{5pt}\fcolorbox{white}{white}{$\displaystyle{\frac{(1-\alpha_{q}\beta_{q}T)(1-\delta_{q}T)(1-\bar{\delta_{q}}T)(1-\bar{\alpha_{q}}\bar{\beta_{q}}T)}{(1-T)(1-qT)^{103}(1-q^2T)^{103}(1-q^3T)}}$}} \\
					\parbox[c]{0.05\hsize}{\centering\setlength{\fboxrule}{1pt}\fcolorbox{gray!25}{gray!25}{$3$}}
					
					& \parbox[c]{0.7\hsize}{\centering\setlength{\fboxrule}{3pt}\fcolorbox{gray!25}{gray!25}{$\displaystyle{{\frac {1}{ \left( 1-T \right)  \left( 1- qT \right) ^
										{340}  \left( {{1-
												\it \bar{\alpha}_q}}^{2}{\it \bar{\beta}_q}\,T \right) \left(1- {{\it \alpha_q}}^{
											2}{\it \beta_q}\,T \right) \left(1- {q}^{2}T \right) ^{1402}\left(1 - {q}^{2}c_qT \right) ^{2}  \left(1- {q}^{3}T \right) ^{340}    
										\left(1- {q}^{4}T \right) }}}$}} \\
					\parbox[c]{0.05\hsize}{\centering\setlength{\fboxrule}{1pt}\fcolorbox{white}{white}{$4$}} &
					\parbox[c]{0.7\hsize}{\centering\setlength{\fboxrule}{3pt}\fcolorbox{white}{white}{$\displaystyle{{\frac { \left( {{1-\it \alpha_q}}^{3}{\it \beta_q}\,T \right)  \left( 1-{q}^{2}{
											\it \delta_q}\,T \right)  \left(1- {q}^{2}{\it \bar{\delta_q}}\,T\right)  \left(1- {{
												\it \bar{\alpha_q}}}^{3}{ \it \bar{\beta_q}}\,T \right) }{\left( 1-T \right)  \left(1- qT
										\right) ^{868} \left(1- {q}^{2}T\right)^{9548}
										\left(1- {q}^{2}{\it c_q}T \right) 
										\left(1- {q}^{3}{\it c_q}T\right) \left(1- {q}^{3}T \right) ^{9548}  \left(1- {q}^{4}T \right) ^{868}\left(1- {q}^{5}T\right)}}}$}}\\
					%	\parbox[c]{0.05\hsize}{\centering\setlength{\fboxrule}{1pt}\fcolorbox{gray!25}{gray!25}{$5$}} &
					%	\parbox[c]{0.70\hsize}{\centering\setlength{\fboxrule}{3pt}\fcolorbox{gray!25}{gray!25}{$\displaystyle{{\frac {1}{ \left( T{q}^{5}-1 \right) ^{1887} \left( T{q}^{6}-1
					%						\right)  \left( T{q}^{3}-1 \right) ^{115520} \left( {{\it \alpha_q}}^{4}{
					%							\it \beta_q}\,T-1 \right)  \left( {{\it \bar{\alpha_q}}}^{4}{\it \bar{\beta_q}}\,T-1 \right) 
					%						\left( {q}^{2}T-1 \right) ^{44520} \left( T{q}^{4}-1 \right) ^{44520}
					%						\left( -1+T \right)  \left( {q}^{3}T{\it c_q}-1 \right) ^{6} \left( qT
					%						-1 \right) ^{1887}}}}$}}
				\end{tabular}
			}
		\end{varwidth}
	\vspace{2mm}
		\captionof{table}{Zeta function of $Y_{6,2}$, $Y_{6,3}$, $Y_{6,4}$}\label{zetaBV6}
	\end{center}
\end{table}

\begin{comment}
Below we also present Hodge diamonds of varieties $Y_{6,2}$, $Y_{6,3}$, $Y_{6,4}.$
\vspace{5mm}
\begin{figure}[htbp!]
	\centering 
	\begin{minipage}[t][3cm][t]{.28\textwidth}
		\centering
		\includegraphics[width=\textwidth]{DiagramsEx-7.pdf}
		\caption*{$H^{**}(Y_{6,2})$}
	\end{minipage}
	\hspace{5cm}
	\begin{minipage}[t][3cm][t]{.28\textwidth}
		\centering
		\includegraphics[width=\textwidth]{DiagramsEx-8.pdf}
		\caption*{$H^{**}(Y_{6,3})$}
	\end{minipage}
\end{figure}

\begin{figure}[htbp!]
	\centering 
\includegraphics[width=0.38\textwidth]{DiagramsEx-9.pdf}
\caption*{$H^{**}(Y_{6,4})$}
\end{figure}

\end{comment}

\printbibliography 
\end{document}